\documentclass[a4paper,11pt]{amsart}

\usepackage[utf8]{inputenc}

\usepackage[a4paper,right=2.5cm,left=2.5cm,bottom=2.5cm, top=2.5cm]{geometry}

\usepackage{latexsym}
\usepackage{graphics}
\usepackage{amsmath}
\usepackage{amssymb}
\usepackage{bm}
\usepackage{dsfont}
\usepackage{mathrsfs}
\usepackage{enumerate,enumitem}
\usepackage{csquotes}
\usepackage{subcaption}
\usepackage{pgfplots,tikz}
\usetikzlibrary{decorations.pathreplacing}
\pgfplotsset{compat=1.10}
\usepgfplotslibrary{fillbetween}
\usetikzlibrary{patterns}
\usetikzlibrary{intersections, calc, angles}
\usepackage{caption} 
\usepackage{soul}

\usepackage{hyperref}
\usepackage{cleveref}
\hypersetup{
	linktocpage=true,
	colorlinks=true,
	linkcolor=blue!70!black,
	citecolor=orange!40!red,
	urlcolor=magenta!80!black,
}
\usepackage[showonlyrefs]{mathtools}

\usepackage{etoolbox,kvoptions,logreq,pdftexcmds}
\numberwithin{equation}{section} 

\newtheorem{theorem}{Theorem}[section]
\newtheorem{corollary}[theorem]{Corollary}
\newtheorem{lemma}[theorem]{Lemma}
\newtheorem{proposition}[theorem]{Proposition}
\theoremstyle{definition} 
\newtheorem{definition}[theorem]{Definition}

\newcommand{\R}{\mathbb{R}}	% Real numbers
\newcommand{\N}{\mathbb{N}} % Natural numbers
\newcommand{\dx}{\,\mathrm{d}x}	% dx
\newcommand{\dy}{\,\mathrm{d}y}	% dy
\newcommand{\ds}{\,\mathrm{d}S}	% ds
\renewcommand{\d}{\,\mathrm{d}}
\newcommand{\weak}{\rightharpoonup}
\newcommand{\nnu}{\bm{\nu}}  % nu exterior normal

\newcommand{\norm}[1]{\left\lVert #1 \right\lVert}
\newcommand{\abs}[1]{\left| #1 \right|}
\newcommand{\sub}{\subseteq}

\newcommand{\tu}[1]{\textup{#1}}

\DeclareMathOperator{\dive}{\mathrm{div}}

\newenvironment{bvp}{\left\{\begin{aligned}  }{\end{aligned}\right.}

\usepackage{amsaddr}
\author[R.~Ognibene]{Roberto Ognibene}
\address[R.~Ognibene]{%
		Università degli Studi di Milano-Bicocca - Dipartimento di Matematica e Applicazioni}
\email[R.~Ognibene]{roberto.ognibene@unimib.it}

\title{\mbox{On asymptotics of Robin eigenvalues in the Dirichlet limit}}

\begin{document}

\begin{abstract}
	We investigate the asymptotic behavior of the eigenvalues of the Laplacian with homogeneous Robin boundary conditions, when the (positive) Robin parameter is diverging. In this framework, since the convergence of the Robin eigenvalues to the Dirichlet ones is known, we address the question of quantifying the rate of such convergence. More precisely, in this work we identify the proper geometric quantity representing (asymptotically) the first term in the expansion of the eigenvalue variation: it is a novel notion of torsional rigidity. Then, by performing a suitable asymptotic analysis of both such quantity and its minimizer, we prove the first-order expansion of any Robin eigenvalue, in the Dirichlet limit. Moreover, the convergence rate of the corresponding eigenfunctions is obtained as well. We remark that all our spectral estimates are explicit and sharp, and cover both the cases of convergence to simple and multiple Dirichlet eigenvalues.
	
\end{abstract}
	
\maketitle

\thispagestyle{empty}

\section{Introduction and main results}

\subsection{Introduction}

Let $d\geq 2$ and let $\Omega\sub\R^d$ be a bounded, open, Lipschitz set. First of all, we introduce the main characters of the present paper, i.e. the Robin eigenvalues. For any $\alpha>0$, we consider the eigenvalue problem for the Laplace operator with homogeneous Robin boundary conditions, that is
\begin{equation}\label{eq:robin_strong}
	\begin{bvp}
		-\Delta\varphi&=\lambda\varphi, &&\text{in }\Omega, \\
		\partial_{\nnu}\varphi+\alpha\varphi&=0, &&\text{on }\partial\Omega,
	\end{bvp}
\end{equation}
with $\nnu$ denoting the unit outer normal vector of $\partial\Omega$, whose weak formulation is the following:
\begin{equation*}
	\left\{\begin{aligned}&\varphi\in H^1(\Omega)\setminus\{0\} \quad\text{and}\quad \lambda\in\R\quad\text{such that} \\
		&\int_{\Omega}\nabla\varphi\cdot\nabla v\dx+\alpha\int_{\partial\Omega}\varphi v\ds=\lambda\int_\Omega \varphi v\dx\quad\text{for all }v\in H^1(\Omega).
	\end{aligned}\right.
\end{equation*}
By classical spectral theory, it is well known that \eqref{eq:robin_strong} admits a diverging sequence of positive eigenvalues, which we denote by
\begin{equation*}
	0<\lambda_1^\alpha\leq \lambda_2^\alpha\leq \dots\leq \lambda_n^\alpha\leq \dots,
\end{equation*}
where we assume that any eigenvalue is repeated according to its multiplicity. 

The Robin eigenvalues, in the range $\alpha>0$, represent a sort of intermediate quantity between the Neumann and the Dirichlet eigenvalues: indeed, if we denote, respectively, by $\{\mu_n\}_{n\in\N}$ and by $\{\lambda_n\}_{n\in\N}$ the eigenvalues of the Neumann and Dirichlet Laplacian (on the same domain $\Omega$, still in increasing order and repeated according to their multiplicities), by the min-max characterization
\begin{equation*}
	\lambda_n^\alpha=\min_{\substack{F\sub H^1(\Omega) \\ \dim F=n}}\max_{\substack{u\in F \\ u\neq 0}}\,\frac{\displaystyle \int_\Omega |\nabla u|^2\dx+\alpha\int_{\partial\Omega}u^2\ds}{\displaystyle \int_\Omega u^2\dx}
\end{equation*}
one can easily see that
\begin{equation*}
	\mu_n\leq \lambda_n^\alpha\leq\lambda_n\quad\text{for all }\alpha>0~\text{and all }n\in\N.
\end{equation*}
Moreover, for any $n\in\N$, again from the variational characterization, it is clear that the function $\alpha\mapsto\lambda_n^\alpha$ is monotone non-decreasing and concave upper semicontinuous (being the infimum of linear functions of the variable $\alpha$). 
At this point, by standard arguments, one can easily prove that
\begin{equation*}
	\lim_{\alpha\to 0}\lambda_n^\alpha=\mu_n 
\end{equation*}
and that
\begin{equation}\label{eq:convergence_eigenvalues}
	\lim_{\alpha\to+\infty}\lambda_n^\alpha=\lambda_n,
\end{equation}
for all $n\in\N$. Our main aim, in the present paper, consists in quantifying the rate of convergence in the latter case (we refer to \cite[Section 4.1]{BS} for the analysis in the former case). We remark that the Robin problem is very rich of properties and there are still lots of open questions in the field: we refer to \cite{BFK_robin} for a complete picture of the state of the art. In particular, we here partially answer to \cite[Open problem 4.7]{BFK_robin}.

The problem we are considering, despite being a basic question in the field, seems not to be treated, in the literature, as much as other aspects (for instance, the case $\alpha\to-\infty$). As far as the author's knowledge is concerned, the main contributions in the case $\alpha\to+\infty$ are due to A. V. Filinovskiy, dated 2011-2017, see \cite{Filinovskiy2014,Filinovskiy2014_a,Filinovskiy2014_b,Filinovskiy2015,Filinovskiy2015_a,Filinovskiy2017}, and to \cite{BBBT}. Let us first describe the $1$-dimensional case, which is essentially explicit and well understood. Following \cite[Section 4.3.1]{BFK_robin} and assuming $\Omega=(0,1)$, we have that problem \eqref{eq:robin_strong} translates into
\begin{equation*}
	\begin{cases}
		-\varphi''=\lambda \varphi, &\text{in }(0,1), \\
		-\varphi'(0)+\alpha \varphi(0)=0, &\\
		\varphi'(1)+\alpha \varphi(1)=0, &
	\end{cases}
\end{equation*}
which can be solved through the equation
\begin{equation*}
	\alpha^2+2\alpha\sqrt{\lambda}\cot\sqrt{\lambda}-\lambda=0.
\end{equation*}
From this, one can deduce the asymptotic behavior of the Robin eigenvalues, that is
\begin{equation}\label{eq:1d}
	\lambda_n^\alpha=n^2\pi^2-\frac{4n^2\pi^2}{\alpha}+\frac{12n^2\pi^2}{\alpha^2}+o\left(\frac{1}{\alpha^2}\right),\quad\text{as }\alpha\to+\infty,
\end{equation}
see \cite[Section 4.3.1]{BFK_robin}. We point out that, since $\lambda_n=n^2\pi^2$ and $\varphi_n(x)=\sqrt{2}\sin(n\pi x)$ is the corresponding $L^2(0,1)$-normalized Dirichlet eigenfunction, we have that
\begin{equation*}
	(\varphi_n'(0))^2+(\varphi_n'(1))^2=4n^2\pi^2,
\end{equation*}
which coincides with the first term in the asymptotic expansion \eqref{eq:1d}. In higher dimensions $d\geq 2$, the sharp asymptotic behavior of $\lambda_n^\alpha$ is known in the case in which $\Omega$ is of class $C^3$ and the corresponding Dirichlet eigenvalue $\lambda_n$ is simple or when $\Omega$ is $C^\infty$ in the general case of possibly multiple limit eigenvalues. More precisely, in \cite{Filinovskiy2017} the author proves that, if $\partial\Omega$ is of class $C^3$ and $\lambda_n$ is simple (and $\varphi_n$ denotes the unique, up to a sign, corresponding $L^2(\Omega)$-normalized eigenfunction), then
\begin{equation}\label{eq:fili}
	\lambda_n^\alpha=\lambda_n-\frac{1}{\alpha}\int_{\partial\Omega}(\partial_{\nnu}\varphi_n)^2\ds+o\left(\frac{1}{\alpha}\right),\quad\text{as }\alpha\to+\infty.
\end{equation}
Moreover, in the same work, as crucial part of the proof of \eqref{eq:fili}, the author provides an estimate on the convergence of eigenfunctions, that is
\begin{equation}\label{eq:fili_2}
	\norm{\varphi_n^\alpha-\varphi_n}_{H^2(\Omega)}\leq \frac{M_n}{\alpha}\quad\text{for all }\alpha>\alpha_n,
\end{equation}
for some $M_n>0$ and $\alpha_n>0$, depending on $d$, $\Omega$ and $n$, where $\varphi_n^\alpha$ is the unique eigenfunction of $\lambda_n^\alpha$ satisfying
\begin{equation*}
	\int_\Omega|\varphi_n^\alpha|^2\dx=1,\quad	\text{and}\quad\int_\Omega\varphi_n^\alpha\varphi_n\dx>0\quad\text{for all }\alpha>\alpha_n.
\end{equation*}
Indeed, since $\lambda_n$ is simple, by continuity, $\lambda_n^\alpha$ is simple as well, for $\alpha$ sufficiently large.
The idea behind the proof of \eqref{eq:fili} in \cite{Filinovskiy2017} is quite simple and deserves to be recalled. First of all, we observe that by De l'Hôpital rule there holds
\begin{equation}\label{eq:hopital}
	\lim_{\alpha\to+\infty}\frac{\lambda_n-\lambda_n^\alpha}{\frac{1}{\alpha}}=\lim_{\alpha\to+\infty}\frac{\frac{\d}{\d \alpha}\lambda_n^\alpha}{\frac{1}{\alpha^2}},
\end{equation}
provided the eigenvalue is differentiable with respect to $\alpha$. Hence, the first step is to compute the derivative of the eigenvalue $\lambda_n^\alpha$ with respect to the parameter $\alpha$. In this direction, once the continuity of eigenvalues and eigenfunctions with respect to $\alpha$ is obtained, by explicit computations the author proves that
\begin{equation}\label{eq:fili_4}
	\frac{\d}{\d\alpha}\lambda_n^\alpha=\int_{\partial\Omega}|\varphi_n^\alpha|^2\ds.
\end{equation}
Now, in view of the Robin boundary conditions satisfied by $\varphi_n^\alpha$, we have
\begin{equation*}
	\frac{\d}{\d\alpha}\lambda_n^\alpha=\frac{1}{\alpha^2}\int_{\partial\Omega}(\partial_{\nnu}\varphi_n^\alpha)^2\ds,
\end{equation*}
which allows to rewrite \eqref{eq:hopital} as
\begin{equation}\label{eq:fili_3}
	\lim_{\alpha\to+\infty}\frac{\lambda_n-\lambda_n^\alpha}{\frac{1}{\alpha}}=\lim_{\alpha\to+\infty}\int_{\partial\Omega}(\partial_{\nnu}\varphi_n^\alpha)^2\ds.
\end{equation}
At this point, the main effort in \cite{Filinovskiy2017} is devoted to the analysis of the behavior of the eigenfunction $\varphi_n^\alpha$ as $\alpha\to+\infty$. In particular, by making a careful use of elliptic regularity estimates, the author proves \eqref{eq:fili_2}, which, in turn implies continuity in $L^2(\partial\Omega)$ of $\partial_{\nnu}\varphi_n^\alpha$ with respect to $\alpha$. Combining this fact with \eqref{eq:fili_3}, one obtains \eqref{eq:fili}.

Nevertheless, it seems arduous to expect that the methods in \cite{Filinovskiy2017} could be extended to the case in which the limit eigenvalue $\lambda_n$ is multiple. Indeed, an ubiquitous (and well known) phenomenon in spectral stability is the lack of differentiability of eigenvalues in case of multiplicity. In particular, the presence of different branches of eigenvalues crossing at the same point produces conical points in the graph $\{(\alpha,\lambda_n^\alpha)\colon \alpha>0\}$. From another point of view, if $\lambda_n^\alpha$ is not simple, then there exists multiple orthogonal corresponding eigenfunctions and this prevents a formula like \eqref{eq:fili_4} (crucial step in \cite{Filinovskiy2017}) to hold true.

The other main contribution in the field is \cite{BBBT}, where the authors tackle the problem of quantifying the Robin-to-Dirichlet spectral stability through a functional analytic approach, in turn based on Kato's methods (quoting \cite[proof of Theorem 5.3]{BBBT}, we refer to \cite[proof of Theorem VIII.2.9]{kato}). This allows them to derive a second-order asymptotic expansion of any Robin eigenvalue as $\alpha\to+\infty$, which also covers the case of possibly multiple limit eigenvalues, albeit under stronger regularity assumptions on the domain. More precisely, in \cite[Theorem 5.3]{BBBT}, the authors prove that if $\Omega$ is of class $C^\infty$ and $n \in \mathbb{N}$ is such that
	\begin{equation*}
		\lambda_{n-1} < \lambda_n = \cdots = \lambda_{n+m-1} < \lambda_{n+m}
	\end{equation*}
	for some $m \geq 1$, then the following expansion holds
	\begin{equation}\label{eq:BBBT}
		\lambda_{n+i-1}^\alpha = \lambda_n - \frac{1}{\alpha} \,\gamma_1^i + \frac{1}{\alpha^2} \,\gamma_2^i + o\left(\frac{1}{\alpha^2}\right), \quad\text{as } \alpha \to +\infty,
	\end{equation}
	for $i = 1, \dots, m$. Regarding the coefficients appearing in \eqref{eq:BBBT}, we have that:
	\begin{itemize}
		\item $\{\gamma_1^i\}_{i=1,\dots,m}$ denotes the set of eigenvalues of the matrix with entries
		\begin{equation*}
			\left(\int_{\partial\Omega}\partial_{\nnu} \varphi_{n+k-1}\,\partial_{\nnu}\varphi_{n+\ell-1}\ds\right)_{k,\ell=1,\dots,m},
		\end{equation*}
		where $\{\varphi_{n+k-1}\}_{k=1,\dots,m}$ is an $L^2(\Omega)$-orthonormal basis associated with the $m$-multiple eigenvalue $\lambda_n$.
		\item The set $\{\gamma_2^i\}_{i=1,\dots,m}$ is again obtained as a family of eigenvalues of a certain matrix, whose explicit expression can be found in \cite[Theorem 5.3]{BBBT}.
	\end{itemize}

In the present paper, we develop a completely new approach, purely variational and based on techniques coming from singular perturbation theory, which allows to treat at the same time the simple and multiple case. Namely, we establish a variational perturbation theory for this framework, by viewing the Robin eigenvalue problem (for large values of the parameter) as a singular perturbation of the Dirichlet one. A drawback of the approaches adopted in \cite{Filinovskiy2017} and \cite{BBBT} is the requirement of strong regularity assumptions on the domain $\Omega$ (assumed to be of class $C^3$ and $C^\infty$, respectively). In particular, this prevents the application of the results to simple non-smooth domains, such as rectangles and parallelepipeds, which are among the easiest and most considered examples due to the fact that the (Dirichlet!) eigenelements are explicitly known. In this direction,  a remarkable feature of our approach is that it allows to work with \emph{rough} domains. In fact, we prove the first order asymptotic expansion of any Robin eigenvalue under the sole Lipschitz regularity assumption on the domain $\Omega$, which guarantees the following basic properties:
\begin{itemize}
	\item the outer unit normal vector $\nnu$ to be defined a.e. in the sense of surface measure;
	\item compactness of the trace embedding $H^1(\Omega)\hookrightarrow L^2(\partial\Omega)$;
	\item the validity of an integration-by-parts formula;
	\item well-posedness of the coefficients in the asymptotic expansion of the eigenvalue variation, that is $\partial_{\nnu}\varphi\in L^2(\partial\Omega)$ for all Dirichlet eigenfunctions $\varphi \in H^1_0(\Omega)$.
\end{itemize}
We also refer to \Cref{subsec:disclaimer} and \Cref{sec:appendix} for further details. 
 
\subsection{Main results}

Let us fix some notation. We recall that, in the limit $\alpha\to+\infty$, we recover the Dirichlet eigenvalue problem for the Laplacian, i.e.
\begin{equation}\label{eq:dir_strong}
	\begin{bvp}
		-\Delta\varphi&=\lambda\varphi, &&\text{in }\Omega, \\
		\varphi&=0, &&\text{on }\partial\Omega,
	\end{bvp}
\end{equation}
whose weak formulation is
\begin{equation*}
\left\{\begin{aligned}
		&\varphi\in H^1_0(\Omega)\setminus\{0\}\quad\text{and}\quad\lambda\in\R\quad\text{such that} \\
		&\int_{\Omega}\nabla\varphi\cdot\nabla v\dx=\lambda\int_\Omega \varphi v\dx\quad\text{for all }v\in H^1_0(\Omega).
\end{aligned}\right.
\end{equation*}
For any $\lambda\in\R$, we denote by
\begin{equation}\label{eq:eigenspace}
	E(\lambda):=\{\varphi\in H^1_0(\Omega)\colon ~\varphi~\text{weakly solves \eqref{eq:dir_strong}}\}.
\end{equation}
the eigenspace corresponding to $\lambda$. We recall that $\dim E(\lambda)$ is called the \emph{multiplicity} of $\lambda$.
In view of the regularity results which we recalled in \Cref{subsec:reg_eigen}, we now state the following crucial property:
\begin{equation}\label{ass:Omega}
	\text{$\partial_{\nnu}\varphi\in L^2(\partial\Omega)$ for any eigenfunction $\varphi\in H^1_0(\Omega)$ of \eqref{eq:dir_strong}.}\tag{I}
\end{equation}
Roughly speaking, the reason behind is that the restriction-to-the-boundary operator induces a unique, well defined, linear, surjective, continuous trace map (for $\epsilon>0$)
\begin{equation*}
	\mathrm{Tr}\,\colon \{u\in H^{1/2}(\Omega)\colon \Delta u\in H^{-3/2+\epsilon}(\Omega) \}\to L^2(\partial\Omega),
\end{equation*}
which can be applied to $\partial_{x_i}\varphi$ for any $i=1,\dots,d$. We refer to \Cref{subsec:reg_eigen} and, in general, to \cite[Section 3]{BGM2022} for the details.

In view of \eqref{ass:Omega}, the bilinear form
\begin{equation}\label{eq:intr_bilinear}
	\begin{alignedat}{2}
	E(\lambda)&\times E(\lambda)&&\longrightarrow \R, \\
	(&\varphi,\psi)&&\longmapsto  \int_{\partial\Omega}\partial_{\nnu}\varphi\,\partial_{\nnu}\psi\ds
\end{alignedat}
\end{equation}
is well defined and is actually a scalar product (see the proof of \Cref{lemma:equiv_norms}). We denote by $\{\lambda_n\}_{n\in\N}$ the sequence of Dirichlet eigenvalues, in increasing order and repeated according to their multiplicities. We now fix, for the rest of the paper, a particular eigenbasis of $L^2(\Omega)$. More precisely, if $n\in\N$ is such that 
\begin{equation*}
	\lambda_{n-1}<\lambda_n=\dots=\lambda_{n+m-1}<\lambda_{n+m}
\end{equation*}
with $m:=\dim E(\lambda_n)$, we assume that
\begin{equation}\label{eq:hp_eigen}
	\begin{aligned}
		&\varphi_{n+i-1}\in E(\lambda_n)\quad\text{for all }i=1,\dots,m, &&\\
		&\int_\Omega\varphi_{n+i-1}\varphi_{n+j-1}\dx=\delta_{ij}\quad\text{for all }i,j=1,\dots,m&& \\
		& \begin{cases}
			&\displaystyle \int_{\partial\Omega}\partial_{\nnu}\varphi_{n+i-1}\,\partial_{\nnu}\varphi_{n+j-1}\ds=\delta_{ij}\int_{\partial\Omega}(\partial_{\nnu}\varphi_{n+i-1})^2\ds \\
			&\displaystyle \text{for all }i,j=1,\dots,m
		\end{cases} \\
		&\begin{cases}
			&\displaystyle	\int_{\partial\Omega}(\partial_{\nnu}\varphi_{n+i-1})^2\ds\geq \int_{\partial\Omega}(\partial_{\nnu}\varphi_{n+j-1})^2\ds \\
			&	\displaystyle\text{for all }i,j=1,\dots,m,~i\leq j.
		\end{cases}
	\end{aligned}
\end{equation}
We emphasize that the third condition amounts to ask (without loss of generality) that the scalar product \eqref{eq:intr_bilinear}, defined on $E(\lambda_n)$, is diagonal, so that its eigenvalues coincide with
\begin{equation*}
	\int_{\partial\Omega}(\partial_{\nnu}\varphi_{n+i-1})^2\ds\quad\text{for all }i=1,\dots,m.
\end{equation*}
We are now ready to state our main result, which provide the sharp asymptotic behavior of any Robin eigenvalue when $\alpha\to+\infty$.
\begin{theorem}\label{thm:main}
	Let $n\in\N$ be such that
	\begin{equation*}
		\lambda_{n-1}<\lambda_n=\dots=\lambda_{n+m-1}<\lambda_{n+m}
	\end{equation*}
	with $m:=\dim E(\lambda_n)$, and let $\{\varphi_{n+i-1}\}_{i=1,\dots,m}$ be chosen in such a way that \eqref{eq:hp_eigen} holds. Then, there holds
	\begin{equation}\label{eq:main_th1}
		\lambda_{n+i-1}^\alpha=\lambda_n-\frac{1}{\alpha }\int_{\partial\Omega}(\partial_{\nnu}\varphi_{n+i-1})^2\ds+o\left(\frac{1}{\alpha}\right)\quad\text{as }\alpha\to+\infty,
	\end{equation}
	for any $i=1,\dots,m$. Moreover, 
	\begin{itemize}
		\item 	if $\Omega$ is of class $C^{1,1}$, then
		\begin{equation*}
			\lambda_{n+i-1}^\alpha=\lambda_n-\frac{1}{\alpha }\int_{\partial\Omega}(\partial_{\nnu}\varphi_{n+i-1})^2\ds+O\left(\frac{1}{\alpha^{4/3}}\right)\quad\text{as }\alpha\to+\infty,
		\end{equation*} 
		for any $i=1,\dots,m$;
		\item 	if $\Omega$ is of class $C^{2,1}$, then
		\begin{equation*}
			\lambda_{n+i-1}^\alpha=\lambda_n-\frac{1}{\alpha }\int_{\partial\Omega}(\partial_{\nnu}\varphi_{n+i-1})^2\ds+O\left(\frac{1}{\alpha^2}\right)\quad\text{as }\alpha\to+\infty,
		\end{equation*} 
		for any $i=1,\dots,m$.
	\end{itemize}
	In all the previous expansions, the reminder term does not depend on $i$.
\end{theorem}

We refer to points (2) and (3) in \Cref{sec:comments} for comments on the sharpness of \Cref{thm:main} and on other possible assumptions on $\Omega$.

Besides the asymptotic analysis of the eigenvalue variation, in the present paper we are able to obtain the sharp asymptotic behavior of the difference between the Dirichlet and Robin eigenfunctions (more precisely, of the spectral projections), as $\alpha\to+\infty$. In order to state such results, we first need a quick preliminary discussion and to fix some more notation.

Similarly to what happens in other situations (see point (3) of \Cref{sec:comments} for more details), the proof of the sharp asymptotic behavior of the Robin eigenvalues passes through the identification of a suitable  quantity which somehow naturally emerges as first order error term in the eigenvalue variation, see \Cref{sec:sketch}. For the problem under consideration, such crucial role is played by a novel notion of torsional rigidity, which we here introduce. 
\begin{definition}\label{def:torsional}
	For any $\alpha>0$ and $f\in L^2(\partial \Omega)$, we define the \emph{$(\alpha,f)$-torsional rigidity of $\partial\Omega$} as follows:
	\begin{equation*}
		T_\alpha(\partial\Omega,f):=-2\inf\left\{ \frac{1}{2}\int_{\Omega}|\nabla u|^2\dx+\frac{\alpha}{2}\int_{\partial\Omega}u^2\ds-\int_{\partial\Omega}fu\ds\colon u\in H^1(\Omega) \right\}.
	\end{equation*}
\end{definition}
By standard variational methods, one can easily prove, for any $f\in L^2(\partial\Omega)$, the existence of a unique function $U_{\partial\Omega,\alpha,f}\in H^1(\Omega)$ achieving $T_\alpha(\partial\Omega,f)$ and satisfying
\begin{equation}\label{eq:torsion_strong}\tag{T}
	\begin{bvp}
		-\Delta U_{\partial\Omega,\alpha,f}&=0, &&\text{in }\Omega, \\
	\partial_{\nnu}U_{\partial\Omega,\alpha,f}+\alpha\,U_{\partial\Omega,\alpha,f}&=f, &&\text{on }\partial\Omega
	\end{bvp}
\end{equation}
in a weak sense, see \Cref{lemma:minimizer}. The name \enquote{torsional rigidity} comes from the fact that, when $f=1$, the quantity $T_\alpha(\partial\Omega,1)$ can be seen as the torsional rigidity of $\partial\Omega$ for the Dirichlet-to-Neumann operator with constant potential $\alpha$.

It turns out that, when $f$ is the normal derivative of a Dirichlet eigenfunction $\varphi_n$, i.e. $f=\partial_{\nnu}\varphi_n$, such minimizer $U_{\partial\Omega,\alpha,\partial_{\nnu}\varphi_n}$ provides the first order approximation of the difference between $\varphi_n$ and its projection onto the eigenspaces corresponding to the Robin eigenvalues converging to $\lambda_n$. We refer to \Cref{sec:sketch} for some insights on the reason why this happens.

Let us now fix some more notation, for $\alpha>0$ and $n\in\N$. In the rest of the paper, we denote by
\begin{equation*}
	E(\lambda_n^\alpha):=\{\varphi\in H^1(\Omega)\colon \varphi~\text{weakly solves \eqref{eq:robin_strong}  with $\lambda=\lambda_n^\alpha$} \}
\end{equation*}
the eigenspace corresponding to the Robin eigenvalue $\lambda_n^\alpha$, and by
\begin{equation}\label{eq:E_alpha}
	\mathcal{E}_\alpha^n:=\bigoplus_{i=1}^{m} E(\lambda_{n+i-1}^\alpha)
\end{equation}
the sum of the eigenspaces corresponding to the family of Robin eigenvalues converging, as $\alpha\to+\infty$, to the same Dirichlet eigenvalue $\lambda_n$, where $m=\dim E(\lambda_n)$ and
\begin{equation*}
	\lambda_{n-1}<\lambda_n=\cdots=\lambda_{n+m-1}<\lambda_{n+m}.
\end{equation*}
Moreover, we let
\begin{equation*}
	\Pi_\alpha^n\colon L^2(\Omega)\to \mathcal{E}_\alpha^n\sub L^2(\Omega)
\end{equation*}
the orthogonal projection onto $\mathcal{E}_\alpha^n$ with respect to the scalar product of $L^2(\Omega)$. Finally, for the sake of brevity, we denote
\begin{equation}\label{eq:H_alpha}
	\norm{u}_{\mathcal{H}_\alpha}:=\left(\int_\Omega|\nabla u|^2\dx+\alpha\int_{\partial\Omega}u^2\ds\right)^{1/2}\quad\text{for }u\in H^1(\Omega).
\end{equation}
We are now ready to state our main result concerning the rate of convergence of the spectral projections.
\begin{theorem}\label{thm:eigenfunctions}
	For any $n\in\N$ there holds
	\begin{equation*}
		\norm{\varphi-\frac{\Pi_\alpha^n\varphi}{\norm{\Pi_\alpha^n\varphi}_{L^2(\Omega)}}-U_{\partial\Omega,\alpha,\partial_{\nnu}\varphi}}_{\mathcal{H}_\alpha}^{2}=o\left(\frac{1}{\alpha}\right)\quad\text{as }\alpha\to+\infty,
	\end{equation*}
	for all $\varphi\in E(\lambda_n)$ such that $\norm{\varphi}_{L^2(\Omega)}=1$, where $U_{\partial\Omega,\alpha,\partial_{\nnu}\varphi}$ denotes the unique function achieving $T_\alpha(\partial\Omega,\partial_{\nnu}\varphi)$. In particular, if $\dim E(\lambda_n)=1$, we have
	\begin{equation*}
		\norm{\varphi_n-\varphi_n^\alpha-U_{\partial\Omega,\alpha,\partial_{\nnu}\varphi_n}}_{\mathcal{H}_\alpha}^{2}=o\left(\frac{1}{\alpha}\right),\quad\text{as }\alpha\to+\infty,
	\end{equation*}
	where $\varphi_n^\alpha$ is the unique $L^2(\Omega)$-normalized eigenfunction of $\lambda_n^\alpha$ such that
	\begin{equation*}
		\int_\Omega\varphi_n\varphi_n^\alpha\dx>0\quad\text{for }\alpha~\text{sufficiently large}.
	\end{equation*}
	Moreover, if $\Omega$ is of class $C^{1,1}$, then the remainder term $o(\alpha^{-1})$ can be replaced by $O(\alpha^{-2})$. All the remainder terms only depend on $d$, $\Omega$ and $n$.
\end{theorem}

We now state, as a consequence of \Cref{thm:eigenfunctions}, the explicit asymptotic behavior of the difference, in norm, between a Dirichlet eigenfunction and its projection onto the corresponding Robin eigenspaces. Loosely speaking, in view of \Cref{thm:eigenfunctions} we have 
\begin{equation*} 
	\norm{\varphi-\frac{\Pi_\alpha^n\varphi}{\norm{\Pi_\alpha^n\varphi}_{L^2(\Omega)}}}_{\mathcal{H}_\alpha}^2\sim\norm{U_{\partial\Omega,\alpha,\partial_{\nnu}\varphi}}_{\mathcal{H}_\alpha}^2\quad\text{as }\alpha\to+\infty
\end{equation*}
and, by \Cref{lemma:minimizer},
\begin{equation*}
	\norm{U_{\partial\Omega,\alpha,\partial_{\nnu}\varphi}}_{\mathcal{H}_\alpha}^2=T_\alpha(\partial\Omega,\partial_{\nnu}\varphi).
\end{equation*}
Therefore, studying the behavior of the spectral projection amounts to study $T_\alpha(\partial\Omega,\partial_{\nnu}\varphi)$, as $\alpha\to+\infty$. Hence, in \Cref{lemma:bound} we prove that
\begin{equation*}
	\alpha\, T_\alpha(\partial\Omega,\partial_{\nnu}\varphi)\to \int_{\partial\Omega}(\partial_{\nnu}\varphi)^2\ds\quad\text{as }\alpha\to+\infty,
\end{equation*}
thus leading to the following.
\begin{corollary}\label{cor:eigenfunctions}
	For any $n\in\N$, we have that
	\begin{equation*}
		\norm{\varphi-\frac{\Pi_\alpha^n\varphi}{\norm{\Pi_\alpha^n\varphi}_{L^2(\Omega)}}}_{\mathcal{H}_\alpha}^2=\frac{1}{\alpha}\int_{\partial\Omega}(\partial_{\nnu}\varphi)^2\ds+o\left(\frac{1}{\alpha}\right)\quad\text{as }\alpha\to+\infty,
	\end{equation*}
	for all $\varphi\in E(\lambda_n)$ such that $\norm{\varphi}_{L^2(\Omega)}=1$. If $\dim E(\lambda_n)=1$, then
	\begin{equation*}
		\norm{\varphi_n-\varphi_n^\alpha}_{\mathcal{H}_\alpha}^2=\frac{1}{\alpha}\int_{\partial\Omega}(\partial_{\nnu}\varphi_n)^2\ds+o\left(\frac{1}{\alpha}\right)\quad\text{as }\alpha\to+\infty,
	\end{equation*}
	where $\varphi_n^\alpha$ is the unique $L^2(\Omega)$-normalized eigenfunction of $\lambda_n^\alpha$ such that
	\begin{equation*}
		\int_\Omega\varphi_n\varphi_n^\alpha\dx>0\quad\text{for }\alpha~\text{sufficiently large}.
	\end{equation*}
	If $\Omega$ is of class $C^{1,1}$, then the remainder term $o(\alpha^{-1})$ can be replaced by $O(\alpha^{-4/3})$ and, if $\Omega$ is of class $C^{2,1}$, by $O(\alpha^{-3/2})$. Moreover, all the remainder terms only depend on $d$, $\Omega$ and $n$.
\end{corollary}

Combining \Cref{cor:eigenfunctions} with \eqref{eq:fili_2} (see \cite[Theorem 4]{Filinovskiy2017}) we deduce that the main contribution in the estimate of the $\mathcal{H}_\alpha$-norm of the difference of Robin and Dirichlet eigenfunctions (at least in case of simple limit eigenvalue), is given by the boundary term, while the $L^2(\Omega)$-norm of the gradient of the difference is a lower order term. More precisely, we have the following.
\begin{corollary}
	Let $\Omega$ be of class $C^3$ and let $n\in\N$ be such that $\lambda_n$ is simple. Then, there holds
	\begin{equation*}
		\int_{\partial\Omega}|\varphi_n^\alpha|^2\ds=\frac{1}{\alpha^2}\int_{\partial\Omega}(\partial_{\nnu}\varphi_n)^2\ds+O\left(\frac{1}{\alpha^{5/2}}\right)\quad\text{as }\alpha\to+\infty,
	\end{equation*}
		where $\varphi_n^\alpha$ is the unique $L^2(\Omega)$-normalized eigenfunction of $\lambda_n^\alpha$ such that
	\begin{equation*}
		\int_\Omega\varphi_n\varphi_n^\alpha\dx>0\quad\text{for }\alpha~\text{sufficiently large}.
	\end{equation*}
\end{corollary}

Finally, as a byproduct of our main result \Cref{thm:main}, we are able to obtain a remarkable property concerning the multiplicity of Robin eigenvalues on rectangles: every Dirichlet eigenvalue on a $2$-dimensional rectangle splits when passing to Robin regime, for $\alpha<+\infty$. More precisely, we have the following.
\begin{corollary}\label{cor:simplicity}
	Let $\Omega:=(0,l)\times (0,L)$, with $l\neq L$. Then, for any $n\in\N$, $\lambda_n^\alpha$ is simple for $\alpha$ sufficiently large (depending on $n$).
\end{corollary}

\subsection{Comments on the main results}\label{sec:comments}

We now briefly discuss on our main results and propose some open problems.

\begin{enumerate}
	\item First of all, we emphasize that the present paper is essentially self-contained and the proofs do not rely on any previous works. The only result which we recall from the literature is the so called \emph{Lemma on small eigenvalues} by Colin de Verdière, i.e. \Cref{lemma:CdV}. However, the proof of it is rather straightforward and, in turn, self-contained, see for instance \cite[Proposition B.1]{ALM2022}.
	\item We remark that all our asymptotic expansions (i.e. \Cref{thm:main} and \Cref{cor:eigenfunctions}) are always sharp, in the sense that
	\[
	\int_{\partial\Omega}(\partial_{\nnu}\varphi)^2\ds>0
	\]
	for any Dirichlet eigenfunction $\varphi\in H^1_0(\Omega)$. Indeed, if $\varphi\equiv\partial_{\nnu}\varphi\equiv 0$ on $\partial\Omega$, standard unique continuation results imply that $\varphi\equiv 0$ in $\Omega$ (see the proof of \Cref{lemma:equiv_norms} for further details).
	\item We point our that \eqref{ass:Omega} are the literally minimal assumptions for \eqref{eq:main_th1} to hold true for any $n\in\N$. Moreover, this emphasizes that the techniques used to prove \eqref{eq:main_th1} do not rely on regularity estimates: indeed, they are purely variational and they just need the objects to be well defined. Essentially, what we need is: discreteness of the spectrum for both the Dirichlet and the Robin Laplacian, possibility of integrating by parts (see \Cref{subsec:dive} for further details) and well-posedness of the first order term in the expansion. We also remark that the $C^{1,1}$ regularity assumption can be replaced by any other assumption which ensures that $E(\lambda_n)\sub H^2(\Omega)$, together with the corresponding elliptic estimates.  For instance, one can replace $C^{1,1}$ regularity with convexity (see e.g. \cite[Theorem 3.2.1.2]{grisvard}) or with a uniform outer ball condition, see \cite{adolfsson}.
	\item The emergence (as perturbation parameter) of a geometric quantity resembling the torsional rigidity of a set is not new. Indeed, such kind of phenomenon manifests in other instances of singularly perturbed eigenvalue problems, such as: mixed Dirichlet-Neumann boundary condition, with small Neumann part (see \cite{AO}), Neumann problems in perforated domains (see \cite{FLO}) and Aharonov-Bohm eigenvalues with moving poles (see \cite{FNOS}).
	\item We remark that the notion of $(\alpha,f)$-torsional rigidity here introduced (see \Cref{def:torsional}) seems to be new in the literature. Nevertheless, one can observe similarities with  the notion of \emph{boundary $\delta$-torsional rigidity} introduced in \cite{brasco} and with the \enquote{classical} \emph{Robin torsional rigidity}, for which we refer, among many others, to \cite{ANT,BG,CF}. In the former case, the corresponding functional is
	\begin{equation*}
		\frac{1}{2}\int_\Omega|\nabla u|^2\dx+\frac{\delta^2}{2}\int_\Omega u^2\dx-\int_{\partial\Omega}u\ds,
	\end{equation*}
	while, in the latter
	\begin{equation*}
		\frac{1}{2}\int_{\Omega}|\nabla u|^2\dx+\frac{\alpha}{2}\int_{\partial\Omega}u^2\ds-\int_\Omega u\dx.
	\end{equation*}

\end{enumerate}

We now leave here some open questions (see also \cite{BFK_robin} for other open problems).
\begin{enumerate}
		\item[(i)] Our approach is rather flexible and we believe that, by defining suitable variants of the torsional rigidity given in \Cref{def:torsional}, one could be able to tackle more general problems. For instance, one can consider an $x$-dependent coefficient in the Robin boundary conditions, i.e.
	\begin{equation*}
		\partial_{\nnu}\varphi+\beta_\alpha(x)\,\varphi=0,\quad\text{on }\partial\Omega,
	\end{equation*}
	with $\beta_\alpha\colon\partial\Omega\to\R_+$, as well as mixed Dirichlet-Robin or Neumann-Robin boundary conditions.
		\item[(ii)] Since the first term in the asymptotic expansion of the eigenvalue variation coincides with the square of the $L^2(\partial\Omega)$-norm of the normal derivative of a Dirichlet eigenfunction (see \Cref{thm:main}), it is clear that some regularity on $\Omega$ must be assumed in order to avoid for such term to be $+\infty$. One can then ask: what happens when the domain $\Omega$ is not regular enough to guarantee that
	\begin{equation*}
		\int_{\partial\Omega}(\partial_{\nnu}\varphi)^2\ds <+\infty
	\end{equation*}
	for any Dirichlet eigenfunction $\varphi\in H^1_0(\Omega)$ ? While the topic of asymptotic behavior of Robin eigenvalues in non-Lipschitz domains has already been approached in the case $\alpha\to-\infty$ (see e.g. \cite{kovarik} and references therein), it seems to be completely open in the case $\alpha\to+\infty$.
		\item[(iii)] With reference to \Cref{cor:eigenfunctions}, an interesting open question would be to understand whether \Cref{cor:simplicity} is sharp or if every Robin eigenvalue on a rectangle is simple above a certain threshold for the parameter $\alpha$ (independent of the index of the eigenvalue). In this sense, it might be useful to recall \cite[Theorem 1.2]{RW2021}, which treats the case when $\alpha\to 0^+$ (i.e. in the Neumann limit) and states that the latter situation cannot happen: namely, the authors prove that for any $\alpha>0$ there exists $\alpha'\in(0,\alpha)$ and $n\in\N$ such that $\lambda_n^{\alpha'}$ is multiple.
\end{enumerate}

\subsection{Disclaimer on Sobolev regularity and integration by parts}\label{subsec:disclaimer}
Before starting, let us make some remarks concerning the regularity of the solutions of the PDEs we deal with in the present paper. To maintain continuity, we here just state the properties and we postpone the details and the proper references to \Cref{sec:appendix}. We recall that $\Omega\sub\R^d$ is an open, bounded, Lipschitz set.

We have that the following:
\begin{itemize}
	\item eigenfunctions of the Laplacian with Dirichlet boundary conditions, i.e. satisfying \eqref{eq:dir_strong};
	\item eigenfunctions of the Laplacian with Robin boundary conditions, i.e. satisfying \eqref{eq:robin_strong};
	\item torsion functions with datum $f\in L^2(\partial\Omega)$, i.e. satisfying \eqref{eq:torsion_strong}
\end{itemize}
are functions belonging to $H^{3/2}(\Omega)$ and whose distributional Laplacian is in $L^2(\Omega)$. As a consequence, their normal derivative on $\partial\Omega$ belongs to $L^2(\partial\Omega)$ and, in particular, the Robin boundary conditions can be seen as identities in $L^2(\partial\Omega)$. Finally, for those functions the divergence theorem holds true (see \eqref{eq:A_int_by_parts}); hence, every time in which, throughout the present work, we say \enquote{by integration by parts}, we fall under the assumptions of \eqref{eq:A_int_by_parts}.

\section{Sketch of the proof}\label{sec:sketch}

The main tool we use in order to prove our main results \Cref{thm:main} and \Cref{thm:eigenfunctions} is the so called \emph{Lemma on small eigenvalues} due to Colin de Verdiére \cite{ColindeV1986}, which we now recall. In particular, we apply the following refinement of the original result, whose proof can be found in \cite[Proposition B.1]{ALM2022} (see also \cite[Lemma 7.1]{FLO} for a focus on the case of simple eigenvalue).

\begin{lemma}[Lemma on small eigenvalues]\label{lemma:CdV}
	Let $(\mathcal{H},(\cdot,\cdot)_{\mathcal{H}})$ be a real
	Hilbert space and let $\mathcal{D}\sub\mathcal{H}$ be a dense
	subspace. Let $q\colon \mathcal{D}\times\mathcal{D}\to\R$ be a
	symmetric bilinear form, such that $q$ is semi-bounded from below, that is
	\begin{equation*}
		\inf_{u\in \mathcal{H}\setminus\{0\}}\frac{q(u,u)}{\norm{u}_{\mathcal{H}}^2}>-\infty,
	\end{equation*}
	$q$ admits an increasing sequence of eigenvalues
	$\{\nu_i\}_{i\in\N}$, 
	and there exists an orthonormal basis $\{g_i\}_{i\in\N}$ of
	$\mathcal H$ such that $g_i\in\mathcal D$ is an eigenvector of $q$
	associated to the eigenvalue $\nu_i$, that is $q(g_i,v)=\nu_i(g_i,v)_{\mathcal{H}}$
	for all $i\geq1$ and $v\in \mathcal D$.	Let $m\in \N\setminus\{0\}$ and $F\sub \mathcal{D}$ be
	a $m$-dimensional subspace of $\mathcal D$.	Assume that there exist $n\in\N$ and
	$\gamma>0$ such that
	\begin{itemize}
		\item[\rm (H1)] $\nu_i\leq -\gamma$ for $i\leq n-1$,
		$|\nu_i|\leq \gamma$ for $i=n,\dots,n+m-1$, and 
		$\nu_i\geq \gamma$ for $i\geq n+m$;
		\item[\rm (H2)]
		$\delta:=\sup\{|q(u,v)|\colon u\in \mathcal{D},~v\in
		F,~\norm{u}_{\mathcal{H}}=\norm{v}_{\mathcal{H}}=1\}<\gamma/\sqrt{2}$.
	\end{itemize}
	If $\{\xi_i\}_{i=1,\dots,m}$ are the eigenvalues (in
	ascending order) of $q$ restricted to $F$, i.e. 
	\begin{equation*}
		\xi_i:=\min_{\substack{G\sub F \\ \dim G=i}}\max_{\substack{u\in G \\ u\neq 0}}\frac{q(u,u)}{\norm{u}_{\mathcal{H}}^2},
	\end{equation*}
	 then we have
	\begin{equation}\label{eq:CdV_th1}
		|\nu_{n+i-1}-\xi_i|\leq\frac{4\delta^2}{\gamma}\quad\text{for all }i=1,\dots,m.
	\end{equation}
	Moreover, if
	$\Pi\colon\mathcal{D}\to
	\mathrm{span}\,\{g_n,\dots,g_{n+m-1}\}$ denotes the
	orthogonal projection onto the subspace of
	$\mathcal D$  spanned by $\{g_n,\dots,g_{n+m-1}\}$,
	we have
	\begin{equation*}
		\frac{\norm{v-\Pi v}_{\mathcal{H}}}{\norm{v}_{\mathcal{H}}}\leq \frac{\sqrt{2}\delta}{\gamma}\quad\text{for all }v\in F.
	\end{equation*}
\end{lemma}
This is an abstract result which provides a quantitative estimate on the \enquote{small eigenvalues} $\{\nu_{n+i-1}\}_{i=1,\dots,m}$ of the bilinear form $q$; hence, in order to apply it in our setting, we first need to properly (and concretely) define all the objects appearing there. We start by choosing the functional spaces and the quadratic form, i.e. we let
\begin{equation*}
	\mathcal{H}:=L^2(\Omega),\qquad \mathcal{D}:=H^1(\Omega)
\end{equation*}
and, for $\alpha>0$ and $n\in\N$ fixed, we let
\begin{equation}\label{eq:q_intr}
	q(u,v):=\int_\Omega\nabla u\cdot\nabla v\dx+\alpha\int_{\partial\Omega}uv\ds-\lambda_n\int_\Omega uv\dx,
\end{equation}
so that we have $\nu_i:=\lambda_i^\alpha-\lambda_n$ for all $i\in\N$. With this choice, if $m=\dim E(\lambda_n)$ is the multiplicity of $\lambda_n$, \Cref{lemma:CdV} allows then to provide quantitative estimates for the \enquote{small eigenvalues}
\begin{equation*}
	\nu_{n+i-1}=\lambda_{n+i-1}^\alpha-\lambda_n\quad\text{for }i=1,\dots,m.
\end{equation*}
The request of smallness of the eigenvalues is encoded in assumption \textit{(H1)}, which is then satisfied, in view of \eqref{eq:convergence_eigenvalues}, with 
	\begin{equation*}
	\gamma:=\frac{1}{2}\min\left\{\lambda_n-\lambda_{n-1},\lambda_{n+m}-\lambda_n \right\}
\end{equation*}
when $n\geq 2$ and 
\begin{equation*}
	\gamma:=\frac{1}{2}(\lambda_2-\lambda_1)
\end{equation*}
when $n=1$, by choosing $\alpha$ sufficiently large (depending on $d$, $\Omega$ and $n$). On the other hand, assumption \textit{(H2)} is closely related with the choice of the \enquote{approximating} eigenspace $F$, which is then crucial. Let us now have a deeper look at this point and, for sake of simplicity, let us assume that $\lambda_n$ is simple, so that $F=\mathrm{Span}\,\{\Phi_\alpha\}$, for some $\Phi_\alpha\in H^1(\Omega)$. In order to have the best estimate for the eigenvalue variation, in view of \eqref{eq:CdV_th1} we want $\delta$ to be as small as possible: hence, let us try to minimize it by suitably choosing $\Phi_\alpha$. By integration by parts, we have
\begin{align*}
	\delta:&=\sup_{u\in H^1(\Omega)}\abs{\int_\Omega\nabla u\cdot\nabla \Phi_\alpha\dx+\alpha\int_{\partial\Omega}u\Phi_\alpha\ds-\lambda_n\int_{\Omega}u\Phi_\alpha\dx}\norm{\Phi_\alpha}_{L^2(\Omega)}^{-1}\norm{u}_{L^2(\Omega)}^{-1} \\
	&=\sup_{u\in H^1(\Omega)}\abs{\int_\Omega u(-\Delta \Phi_\alpha-\lambda_n \Phi_\alpha)\dx+\alpha\int_{\partial\Omega}u\,(\partial_{\nnu}\Phi_\alpha+\alpha \Phi_\alpha)\ds}\norm{\Phi_\alpha}_{L^2(\Omega)}^{-1}\norm{u}_{L^2(\Omega)}^{-1}.
\end{align*}
Therefore, it is clear that for $\Phi_\alpha$ to be the best choice (so that $\delta=0$), it should satisfy
\begin{equation}\label{eq:plan_1}
	\begin{bvp}
		-\Delta\Phi_\alpha &=\lambda_n \Phi_\alpha, &&\text{in }\Omega, \\
		\partial_{\nnu}\Phi_\alpha +\alpha\Phi_\alpha &=0, &&\text{on }\partial\Omega.
	\end{bvp}
\end{equation}
However, this is not possible: indeed, since $\alpha\mapsto \lambda_n^\alpha$ is a strictly increasing function (see \eqref{eq:fili_4}), we have that $\lambda_n$ cannot belong to the spectrum of the Robin Laplacian, for $\alpha$ sufficiently large. Hence, the best choice turns into finding a function $\Phi_\alpha$ which \enquote{almost} satisfies \eqref{eq:plan_1}, with the least possible error. A reasonable formulation of such a question is the following
\begin{equation*}
	\min\left\{\frac{1}{2}\int_\Omega|\nabla (\varphi_n-v)|^2+\frac{\alpha}{2}\int_{\partial\Omega}v^2\ds+\int_{\partial\Omega}v\,\partial_{\nnu}\varphi_n\ds\colon v\in H^1(\Omega)\right\}
\end{equation*}
Indeed, if $\Phi_\alpha$ denotes the minimizer of the problem above, then, on one hand, the first term of the functional forces $\Phi_\alpha$ to be close in norm to $\varphi_n$ (eigenfunction of $\lambda_n$), while, on the other hand, the boundary terms force $\Phi_\alpha$ to satisfy the homogeneous Robin boundary condition, i.e.
\begin{equation}\label{eq:plan_2}
	\begin{bvp}
		-\Delta \Phi_\alpha &=\lambda_n\varphi_n, &&\text{in }\Omega, \\
		\partial_{\nnu}\Phi_\alpha+\alpha\Phi_\alpha &=0, &&\text{on }\partial\Omega,
	\end{bvp}
\end{equation}
thus making it a good approximating solution of \eqref{eq:plan_1}. If we now let $U_\alpha:=\varphi_n-\Phi_\alpha$, we have that $U_\alpha$ minimizes
\begin{equation*}
	\min\left\{ \frac{1}{2}\int_\Omega|\nabla u|^2\dx+\frac{\alpha}{2}\int_{\partial\Omega}u^2\ds-\int_{\partial\Omega}u\,\partial_{\nnu}\varphi_n\ds\colon u\in H^1(\Omega) \right\},
\end{equation*}
thus coinciding with the unique function achieving $T_\alpha(\partial\Omega,\partial_{\nnu}\varphi_n)$, i.e. $U_\alpha=U_{\partial\Omega,\alpha,\partial_{\nnu}\varphi_n}$. Let us now investigate which estimate we are able to obtain from \eqref{eq:CdV_th1} by choosing $F=\mathrm{Span}\,\{\Phi_\alpha\}$ and $\Phi_\alpha$ as above. First, by direct computations and Cauchy-Schwarz inequality, we can estimate $\delta$ as follows
\begin{equation*}
	\delta\leq \lambda_n\frac{\norm{U_\alpha}_{L^2(\Omega)}}{\norm{\varphi_n-U_\alpha}_{L^2(\Omega)}}\leq \lambda_n\frac{\norm{U_\alpha}_{L^2(\Omega)}}{1-\norm{U_\alpha}_{L^2(\Omega)}}.
\end{equation*}
Second, for what concerns the family $\{\xi_i\}_{i=1,\dots,m}$, since we are assuming $\lambda_n$ to be simple (i.e. $m=1$), in view of the equation satisfied by $\Phi_\alpha$, \eqref{eq:plan_2}, we can easily compute the restriction to $F$ of the quadratic form $q$, which gives that
\begin{equation*}
	\xi_1:=\frac{q(\Phi_\alpha,\Phi_\alpha)}{\norm{\Phi_\alpha}_{L^2(\Omega)}^2}=-\left(T_\alpha(\partial\Omega,\partial_{\nnu}\varphi_n)+\lambda_n\norm{U_\alpha}_{L^2(\Omega)}^2\right)\norm{\varphi_n-U_\alpha}_{L^2(\Omega)}^{-2}.
\end{equation*}
Summing up, from \eqref{eq:CdV_th1} we deduce that
\begin{equation}\label{eq:plan_3}
	\abs{\lambda_n^\alpha-\lambda_n+\left(T_\alpha(\partial\Omega,\partial_{\nnu}\varphi_n)+\lambda_n\norm{U_\alpha}_{L^2(\Omega)}^2\right)\norm{\varphi_n-U_\alpha}_{L^2(\Omega)}^{-2}}\leq \frac{4\lambda_n^{2}}{\gamma}\frac{\norm{U_\alpha}_{L^2(\Omega)}^2}{\left(1-\norm{U_\alpha}_{L^2(\Omega)}\right)^2},
\end{equation}
for $\alpha$ sufficiently large. It is now clear that the detection of the first order expansion of the eigenvalue variation $\lambda_n-\lambda_n^\alpha$, as $\alpha\to +\infty$, boils down to a careful asymptotic analysis of
\begin{equation*}
	T_\alpha(\partial\Omega,\partial_{\nnu}\varphi_n)\quad\text{and}\quad \norm{U_\alpha}_{L^2(\Omega)},
\end{equation*}
as $\alpha\to +\infty$, which is the core technical part of the present work. In particular, we have
\begin{equation}\label{eq:plan_4}
	\norm{U_\alpha}_{L^2(\Omega)}^2=o(T_\alpha(\partial\Omega,\partial_{\nnu}\varphi_n))
\end{equation}
and
\begin{equation}\label{eq:plan_5}
	T_\alpha(\partial\Omega,\partial_{\nnu}\varphi_n)=\frac{1}{\alpha}\int_{\partial\Omega}(\partial_{\nnu}\varphi_n)^2\ds+o\left(\frac{1}{\alpha}\right),
\end{equation}
as $\alpha\to+\infty$, see \Cref{lemma:L2_weak} and \eqref{eq:bound_th2} respectively (where we deal with the more general case of multiple eigenvalues). We also point out that the asymptotic analysis of these two quantities is closely related with estimating the remainder terms $\omega_n(\alpha)$ and $\rho_n(\alpha)$, defined in \eqref{eq:def_omega} and \eqref{eq:def_rho}, respectively; however, this will be more clear later on in the paper, see \eqref{eq:main_1}. Combining these two facts \eqref{eq:plan_4} and \eqref{eq:plan_5} with \eqref{eq:plan_3}, we obtain the validity of \Cref{thm:main}. Indeed, from \eqref{eq:plan_4} and \eqref{eq:plan_5} (and normalization of $\varphi_n$) we obtain
	\begin{equation*}
		\norm{U_\alpha}_{L^2(\Omega)}^2=o\left(\frac{1}{\alpha}\right)\quad\text{and}\quad\norm{\varphi_n-U_\alpha}_{L^2(\Omega)}^2=1+\norm{U_\alpha}_{L^2(\Omega)}^2-2(\varphi_n,U_\alpha)_{L^2(\Omega)}=1+o(1),
	\end{equation*}
	as $\alpha\to+\infty$. Plugging this and \eqref{eq:plan_5} into \eqref{eq:plan_3} yields
	\begin{equation*}
		\abs{\lambda_n^\alpha-\lambda_n+\left(\frac{1}{\alpha}\int_{\partial\Omega}(\partial_{\nnu}\varphi_n)^2\ds+o\left(\frac{1}{\alpha}\right)\right)(1+o(1))}=o\left(\frac{1}{\alpha}\right)\quad\text{as }\alpha\to+\infty,
	\end{equation*}
	thus implying \Cref{thm:main} in case of simple eigenvalues.

\section{\texorpdfstring{Preliminaries on $T_\alpha(\partial\Omega,f)$}{Preliminaries of the boundary torsional rigidity}}

In the present section, we state some preliminary results concerning the $(\alpha,f)$-torsional rigidity of $\partial\Omega$, introduced in \Cref{def:torsional}, with a particular focus on its asymptotic behavior as $\alpha\to+\infty$. 

Starting with the basics, we investigate existence and uniqueness of the minimizer and the equation which satisfies, after recalling the following standard result (see e.g. \cite[Lemma 5.22]{Lieberman}) 
\begin{lemma}[Coercivity]\label{lemma:coerc}
	There exists $C_{\textup{coerc}}>0$ such that
	\begin{equation*}
		\int_\Omega u^2\dx\leq C_{\textup{coerc}}\left(\int_{\Omega}|\nabla u|^2\dx+\int_{\partial\Omega} u^2\ds\right)\quad\text{for all }u\in H^1(\Omega).
	\end{equation*}
\end{lemma}

We are now able to obtain the following.
\begin{lemma}\label{lemma:minimizer}
	Let $\alpha>0$ and $f\in L^2(\partial\Omega)$. There exists a unique $U_{\partial\Omega,\alpha,f}\in H^1(\Omega)$ achieving
\begin{equation*}
		T_\alpha(\partial\Omega,f)=-2\inf\left\{ \frac{1}{2}\int_{\Omega}|\nabla u|^2\dx+\frac{\alpha}{2}\int_{\partial\Omega}u^2\ds-\int_{\partial\Omega}fu\ds\colon u\in H^1(\Omega) \right\}
	\end{equation*}
	which weakly solves
	\begin{equation*}
		\begin{bvp}
			-\Delta U_{\partial\Omega,\alpha,f}&=0,&&\text{in }\Omega, \\
			\partial_{\nnu}U_{\partial\Omega,\alpha,f}+\alpha U_{\partial\Omega,\alpha,f}&=f,&&\text{on }\partial\Omega,
		\end{bvp}
	\end{equation*}
	in the sense that
	\begin{equation}\label{eq:minimizer_th2}
		\int_\Omega\nabla U_{\partial\Omega,\alpha,f}\cdot\nabla v\dx+\alpha\int_{\partial\Omega}U_{\partial\Omega,\alpha,f}v\ds=\int_{\partial\Omega}fv\ds\quad\text{for all }v\in H^1(\Omega).
	\end{equation}
	Moreover, the map $f\mapsto U_{\partial\Omega,\alpha,f}$ is linear and we have that
	\begin{equation}
		T_\alpha(\partial\Omega,f)=\int_\Omega|\nabla U_{\partial\Omega,\alpha,f}|^2\dx+\alpha\int_{\partial\Omega}|U_{\partial\Omega,\alpha,f}|^2=\int_{\partial\Omega}f U_{\partial\Omega,\alpha,f}\ds \label{eq:minimizer_th1}
	\end{equation}
	and that the following equivalent formulation holds
	\begin{equation}
		T_\alpha(\partial\Omega,f)=\max\left\{\frac{\displaystyle\left(\int_{\partial\Omega}fu\ds\right)^2}{\displaystyle\int_\Omega|\nabla u|^2\dx+\alpha\int_{\partial\Omega}u^2\ds}\colon u\in H^1(\Omega)\setminus\{0\}\right\}. \label{eq:equiv_torsional}
	\end{equation}
\end{lemma}
\begin{proof}
	Let us consider the bilinear form
	\begin{equation*}
		a(u,v):=\int_\Omega\nabla u\cdot\nabla v\dx+\alpha\int_{\partial\Omega}uv\ds\quad\text{defined for }u,v\in H^1(\Omega).
	\end{equation*}
	In view of \Cref{lemma:coerc} and of the trace embedding $H^1(\Omega)\hookrightarrow L^2(\partial \Omega)$, $a(\cdot,\cdot)$ is continuous and coercive and the functional
	\begin{equation*}
		L_f(u):=\int_{\partial\Omega}fu\ds\quad\text{defined for }u\in H^1(\Omega)
	\end{equation*}
	is linear and continuous. Hence, by Lax-Milgram Lemma we conclude the first part of the proof, while linearity of the map $f\mapsto U_{\partial\Omega,\alpha,f}$ is a trivial consequence. Second, \eqref{eq:minimizer_th1} can be obtained by testing \eqref{eq:minimizer_th2} with $U_{\partial\Omega,\alpha,f}$, while \eqref{eq:equiv_torsional} follows from the fact that
	\begin{equation*}
		T_\alpha(\partial\Omega,f)=-2\min_{u\in H^1(\Omega)}\min_{t\in\R}\left\{ \frac{t^2}{2}\int_\Omega|\nabla u|^2\dx+\frac{\alpha t^2}{2}\int_{\partial\Omega}u^2\ds-t\int_{\partial\Omega}fu\ds \right\}
	\end{equation*}
	and that the inner functional is minimized, for $u\neq 0$, at 
	\[
		t= \left(\int_{\partial\Omega}fu\ds\right) \left(\int_{\Omega}|\nabla u|^2\dx+\alpha\int_{\partial\Omega}u^2\ds\right)^{-1}.
	\]
	In fact, if $f\neq 0$ then necessarily $U_{\partial\Omega,\alpha,f}\neq 0$; indeed, if $U_{\partial\Omega,\alpha,f}= 0$, then by \eqref{eq:minimizer_th2} we would have that
	\begin{equation*}
		\int_{\partial\Omega}fv\ds=0\quad\text{for all }v\in H^1(\Omega),
	\end{equation*}	
	which would imply that $f=0$. This completes the proof.
\end{proof}

Even though it is not needed for the proof of the main theorems of the present paper, in the following we prove $C^1$ regularity of the map $\alpha\mapsto T_\alpha(\partial\Omega,f)$ and we explicitly compute its first derivative, which leads to a monotonicity property. We think this is of independent interest and lies as a basic result in possible further investigations of $T_\alpha(\partial\Omega,f)$ as a geometric quantity.

\begin{lemma}\label{lemma:monotonicity}
	For any $f\in L^2(\partial\Omega)$, the function
	\begin{align*}
		(0,+\infty)\ &\to H^1(\Omega), \\
		\alpha&\mapsto U_{\partial\Omega,\alpha,f}
	\end{align*}
	is continuous in $(0,+\infty)$. Moreover, the functions $\alpha\mapsto T_\alpha(\partial\Omega,f)$ and $\alpha\mapsto \alpha \,T_\alpha(\partial\Omega,f)$ are of class $C^1$ and satisfy
	\begin{equation}\label{eq:monotoniciy_th1}
		\frac{\d}{\d\alpha}T_\alpha(\partial\Omega,f)=-\int_{\partial\Omega}U_{\partial\Omega,\alpha,f}^2\ds\leq 0\quad\text{and}\quad \frac{\d}{\d\alpha}(\alpha \,T_\alpha(\partial\Omega,f)) = \int_\Omega |\nabla U_{\partial\Omega,\alpha,f}|^2\dx\geq 0
	\end{equation}
\end{lemma}
\begin{proof}
	For sake of simplicity, in this proof we denote $U_\alpha:=U_{\partial\Omega,\alpha,f}$ and $T_\alpha:=T_\alpha(\partial\Omega,f)$. First of all, we prove continuity in $H^1(\Omega)$ of the map $\alpha\mapsto U_\alpha$. For fixed $\alpha>0$, we consider $\beta>0$ arbitrarily close to $\alpha$. If we let $V_\beta:=U_\beta-U_\alpha$, since it weakly satisfies
	\begin{equation*}
		\begin{bvp}
			-\Delta V_\beta&=0,&&\text{in }\Omega, \\
			\partial_{\nnu} V_\beta+\beta V_\beta&=(\alpha-\beta)U_\alpha, &&\text{on }\partial\Omega,
		\end{bvp}
	\end{equation*}
	by testing with $V_\beta$ itself, we deduce that
	\begin{equation*}
		\int_\Omega|\nabla V_\beta|^2\dx+\beta\int_{\partial\Omega}V_\beta^2\ds=(\alpha-\beta)\int_{\partial\Omega}U_\alpha V_\beta\ds.
	\end{equation*}
	Then, combining this with \Cref{lemma:coerc} we have that
	\begin{equation*}
		\norm{V_\beta}_{H^1(\Omega)}^2\leq C_\beta\left(\int_\Omega|\nabla V_\beta|^2\dx+\beta\int_{\partial\Omega}V_\beta^2\ds\right)= C_	\beta(\alpha-\beta)\int_{\partial\Omega}U_\alpha V_\beta\ds,
	\end{equation*}
	where
	\begin{equation*}
		C_\beta:=\begin{cases}
			\beta^{-1}(C_{\tu{coerc}}+1),&\text{if }\beta\in(0,1], \\
			C_{\tu{coerc}}+1,&\text{if }\beta\in(1,+\infty)
		\end{cases}
	\end{equation*}
	and $C_{\tu{coerc}}>0$ is as in \Cref{lemma:coerc}. Being $\beta$ arbitrarily close to $\alpha>0$, we can actually assume that $C_\beta\leq C_\alpha$, for some $C_\alpha>0$ depending on $d$, $\Omega$ and $\alpha$. Now, by applying the trace inequality, we easily obtain that
	\begin{equation*}
		\norm{U_\beta-U_\alpha}_{H^1(\Omega)}=\norm{V_\beta}_{H^1(\Omega)}\leq C_\alpha|\alpha-\beta|\norm{U_\alpha}_{L^2(\partial\Omega)},
	\end{equation*}
	for some other $C_\alpha>0$ (again depending only on $d$, $\Omega$ and $\alpha$) from which (letting $\beta\to\alpha$) we deduce the desired continuity. Let us now prove that $T_\alpha$ is differentiable and \eqref{eq:monotoniciy_th1} holds.  For any $\alpha,\beta>0$ (sufficiently close to each other), if we test with $U_\beta\in H^1(\Omega)$ the equation satisfied by $U_\alpha$ (and we use \eqref{eq:minimizer_th1}) we get
	\begin{equation*}
		\int_{\Omega}\nabla U_\beta\cdot\nabla U_\alpha\dx+\beta\int_{\partial\Omega}U_\beta U_\alpha=\int_{\partial\Omega}fU_\alpha=T_\alpha.
	\end{equation*}
	Analogously, if we test with $U_\alpha\in H^1(\Omega)$ the equation satisfied by $U_\beta$ (and we again use \eqref{eq:minimizer_th1}) we get
	\begin{equation*}
		\int_{\Omega}\nabla U_\alpha\cdot\nabla U_\beta\dx+\alpha\int_{\partial\Omega}U_\alpha U_\beta=\int_{\partial\Omega}fU_\beta=T_\beta.
	\end{equation*}
	By taking the difference, we obtain
	\begin{equation*}
		\frac{T_\alpha-T_\beta}{\alpha-\beta}=-\int_{\partial\Omega}U_\alpha U_\beta\ds.
	\end{equation*}
	Then, taking the limit as $\beta\to\alpha$, in view of the continuity result of the first part, we derive that
	\begin{equation*}
		\frac{\d}{\d\alpha}T_\alpha=-\int_{\partial\Omega}U_\alpha^2\ds.
	\end{equation*}
	Therefore
	\begin{equation*}
		\frac{\d}{\d\alpha}(\alpha T_\alpha)=T_\alpha+\alpha \frac{\d}{\d\alpha}T_\alpha=\int_{\Omega}|\nabla U_\alpha|^2\dx+\alpha\int_{\partial\Omega}U_\alpha^2\ds+\alpha \frac{\d}{\d\alpha} T_\alpha=\int_{\Omega}|\nabla U_\alpha|^2\dx,
	\end{equation*}
	thus concluding the proof.
	
\end{proof}

The following result contains explicit estimates for the $(\alpha,f)$-torsional rigidity of $\partial\Omega$ and its minimizer, which provide, as a consequence, their sharp asymptotic behavior as $\alpha\to+\infty$.
\begin{lemma}\label{lemma:bound}
	Let $\alpha>0$ and $f\in L^2(\partial\Omega)$. Then
	\begin{equation}\label{eq:bound_th1}
		\frac{\displaystyle\left(\int_{\partial\Omega}f\ds\right)^2}{\mathcal{H}^{d-1}(\partial\Omega)}\leq  \alpha T_\alpha(\partial\Omega,f)\leq \norm{f}_{L^2(\partial\Omega)}^2.
	\end{equation}
	In addition, if $f\in  H^{1/2}(\partial\Omega)\setminus\{0\}$ we have that
	\begin{equation}\label{eq:bound_th2}
		\frac{\displaystyle \alpha\norm{f}_{L^2(\partial\Omega)}^4}{\displaystyle \norm{\nabla V_f}_{L^2(\Omega)}^2+\alpha\norm{f}_{L^2(\partial\Omega)}^2}\leq \alpha T_\alpha(\partial\Omega,f)\leq \norm{f}_{L^2(\partial\Omega)}^2,
	\end{equation}
	where $V_f\in H^1(\Omega)$ denotes the harmonic extension of $f$ in $\Omega$. Moreover, if we let $u_\alpha:=\alpha U_{\partial\Omega,\alpha,f}$ then we have that
	\begin{equation}\label{eq:bound_th4}
		\norm{u_\alpha-f}_{L^2(\partial\Omega)}^2\leq\frac{4\norm{f}_{L^2(\partial\Omega)}^{\frac{2}{3}}\norm{\nabla V_f}_{L^2(\Omega)}^{\frac{4}{3}}}{\alpha^{\frac{2}{3}}}.
	\end{equation}
\end{lemma}
\begin{proof}
	The left bound in \eqref{eq:bound_th1} can be obtained by testing \eqref{eq:equiv_torsional} with a constant function, while the right one follows by applying Cauchy-Schwarz inequality in the numerator of \eqref{eq:equiv_torsional}. Similarly, the left bound in \eqref{eq:bound_th2} comes by testing \eqref{eq:equiv_torsional} with the harmonic extension  $V_f\in H^1(\Omega)$ of $f$. 
	Finally, in order to prove \eqref{eq:bound_th4}, we first observe that rephrasing the right bound of \eqref{eq:bound_th2} in view of \eqref{eq:minimizer_th1} yields
	\begin{equation}\label{eq:bound_2}
		\frac{1}{\alpha}\int_\Omega|\nabla u_\alpha|^2\dx+\int_{\partial\Omega}u_\alpha^2\ds\leq \int_{\partial\Omega}f^2\ds.
	\end{equation}
	Taking into account the the previous inequality and testing with $V_f$ the equation satisfied by $u_\alpha$, we obtain
	\begin{align*}
		\int_{\partial\Omega}(u_\alpha-f)^2\ds&=\int_{\partial\Omega}u_\alpha^2\ds+\int_{\partial\Omega}f^2\ds-2\int_{\partial\Omega}u_\alpha f\ds \notag\\
		&=\int_{\partial\Omega}u_\alpha^2\ds-\int_{\partial\Omega}f^2\ds+2\int_{\partial\Omega}f(f-u_\alpha)\ds \notag\\
		&\leq \frac{2}{\alpha}\int_\Omega\nabla u_\alpha\cdot\nabla V_f\dx\leq \frac{2}{\alpha}\left(\int_\Omega|\nabla u_\alpha|^2\dx\right)^{1/2}\left(\int_\Omega|\nabla V_f|^2\dx\right)^{1/2},
	\end{align*}
	where in the last step we used Cauchy-Schwarz inequality. Combining the previous estimate with the following
	\begin{align*}
		\norm{\nabla u_\alpha}_{L^2(\Omega)}\leq \sqrt{\alpha(\norm{f}_{L^2(\partial\Omega)}^2-\norm{u_\alpha}_{L^2(\partial\Omega)}^2)}&\leq\sqrt{2\alpha\norm{f}_{L^2(\partial\Omega)}}\sqrt{\norm{f}_{L^2(\partial\Omega)}-\norm{u_\alpha}_{L^2(\partial\Omega)}}\\
		&\leq \sqrt{2\alpha\norm{f}_{L^2(\partial\Omega)}}\sqrt{\norm{f-u_\alpha}_{L^2(\partial\Omega)}},
	\end{align*}
	which is again a straightforward consequence of \eqref{eq:bound_2}, we deduce that
	\begin{equation*}
		\norm{u_\alpha-f}_{L^2(\partial\Omega)}^2\leq \frac{2\norm{\nabla V_f}_{L^2(\Omega)}}{\alpha}\sqrt{2\alpha\norm{f}_{L^2(\partial\Omega)}}\sqrt{\norm{f-u_\alpha}_{L^2(\partial\Omega)}}
	\end{equation*}
	Rearranging the terms then leads to \eqref{eq:bound_th4} and concludes the proof.
	
\end{proof}

The quantitative estimates contained in \Cref{sec:quantitative} need the control of the behavior of the torsion function $U_{\partial\Omega,\alpha,f}$ in the case in which also the function $f$ is varying with $\alpha$. Therefore, we now prove a series of three convergence results, in increasing order of regularity. We start with the $L^2(\partial\Omega)$ case.

\begin{lemma}[First order expansion in $L^2(\partial\Omega)$]\label{lemma:weak}
	Let $\{f_\alpha\}_{\alpha>0}\sub L^2(\partial\Omega)$ and $f\in L^2(\partial\Omega)$ be such that
	\begin{equation}\label{eq:weak_hp1}
		f_\alpha\weak f\quad\text{weakly in }L^2(\partial\Omega),~\text{as }\alpha\to+\infty.
	\end{equation}
	Then
	\begin{equation}\label{eq:weak_th1}
		\alpha U_{\partial\Omega,\alpha,f_\alpha}\weak f\quad\text{weakly in }L^2(\partial\Omega),~\text{as }\alpha\to+\infty.
	\end{equation}
	Furthermore, if 
	\begin{equation}\label{eq:weak_hp2}
		f_\alpha\to f\quad\text{strongly in }L^2(\partial\Omega),~\text{as }\alpha\to+\infty,
	\end{equation}
	then
	\begin{equation}\label{eq:weak_th2}
		\alpha\, T_\alpha(\partial\Omega,f_\alpha)\to \int_{\partial\Omega}f^2\ds,\quad\text{as }\alpha\to+\infty.
	\end{equation}
\end{lemma}
\begin{proof}
	For sake of simplicity, in this proof we denote $v_\alpha:=\sqrt{\alpha}U_{\partial\Omega,\alpha,f_\alpha}$ and $u_\alpha:=\alpha U_{\partial\Omega,\alpha,f_\alpha}$. First of all, we observe that, thanks to the right bound of \eqref{eq:bound_th1} and \eqref{eq:minimizer_th1}, we have
	\begin{equation*}
		\int_\Omega|\nabla v_\alpha|^2\dx+\alpha\int_{\partial\Omega}v_\alpha^2\ds\leq \int_{\partial\Omega}f_\alpha^2\ds\quad\text{for all }\alpha>0.
	\end{equation*}
	Now, in view of the equation satisfied by $u_\alpha$ (see \eqref{eq:minimizer_th2}), Cauchy-Schwarz inequality and the previous estimate, we have
	\begin{align*}
		\abs{\int_{\partial\Omega}(u_\alpha-f)v\ds}&=\abs{\int_{\partial\Omega}(f_\alpha-f)v\ds-\frac{1}{\alpha}\int_\Omega\nabla u_\alpha\cdot\nabla v\dx }\\
		&=\abs{\int_{\partial\Omega}(f_\alpha-f)v\ds-\frac{1}{\sqrt{\alpha}}\int_\Omega\nabla v_\alpha\cdot\nabla v\dx} \\
		&\leq \abs{\int_{\partial\Omega}(f_\alpha-f)v\ds }+\frac{1}{\sqrt{\alpha}} \norm{f_\alpha}_{L^2(\partial\Omega)}\norm{\nabla v}_{L^2(\Omega)}
	\end{align*}
	for all $v\in H^1(\Omega)$. Hence, in view of \eqref{eq:weak_hp1} we have \eqref{eq:weak_th1}. Finally, since
	\begin{equation*}
		\alpha\, T_\alpha(\partial\Omega,f_\alpha)=\int_{\partial\Omega}u_\alpha f_\alpha\ds,
	\end{equation*}
	in view of the \eqref{eq:weak_hp2} and \eqref{eq:weak_th1}, we have \eqref{eq:weak_th2}, thus concluding the proof. 
\end{proof}

We now further assume $H^{1/2}(\partial\Omega)$ convergence, which leads to a stronger convergence of the torsion function.
\begin{proposition}[First order expansion in $H^1(\Omega)$]\label{prop:conv_torsion}
	Let $\{f_\alpha\}_{\alpha>0}\sub H^{1/2}(\partial\Omega)$ and $f\in H^{1/2}(\partial\Omega)$ be such that
	\begin{equation*}
		f_\alpha\weak f\quad\text{weakly in }H^{1/2}(\partial\Omega)~\text{as }\alpha\to+\infty.
	\end{equation*}
	Then
	\begin{equation*}
		\alpha U_{\partial\Omega,\alpha,f_\alpha}\weak V_f\quad\text{weakly in }H^1(\Omega)~\text{as }\alpha\to+\infty,
	\end{equation*}
	where $V_f$ is the harmonic extension of $f$ in $\Omega$.
\end{proposition}
\begin{proof}
	In the whole proof, we denote by $C>0$ a positive constant depending only on $d$ and $\Omega$, which may change from time to time. We also denote $u_\alpha:=\alpha U_{\partial\Omega,\alpha,f_\alpha}$, which satisfies
	\begin{equation}\label{eq:conv_1}
		\begin{bvp}
			-\Delta u_\alpha&=0, &&\text{in }\Omega, \\
			\frac{1}{\alpha}\partial_{\nnu}u_\alpha+u_\alpha&=f_\alpha, &&\text{on }\partial\Omega,
		\end{bvp}
	\end{equation}
	in a weak sense. We now recall that the space
	\begin{equation*}
		H_{\textup{div}}(\Omega):=\{\bm{u}\in L^1_{\textup{loc}}(\Omega)^{d}\colon\norm{\bm{u}}_{H_{\textup{div}}(\Omega)}:=\norm{\bm{u}}_{L^2(\Omega)}+\norm{\dive\bm{u}}_{L^2(\Omega)}<+\infty \}
	\end{equation*}
	is continuously embedded into $H^{-1/2}(\partial\Omega)$, which implies that there exists a constant $C_{\text{div}}>0$, depending only on $d$ and $\Omega$, such that
	\begin{equation}\label{eq:conv_2}
		\norm{\partial_{\nnu}u_\alpha}_{H^{-1/2}(\partial\Omega)}\leq C_{\textup{div}}\left(\norm{\nabla u_\alpha}_{L^2(\Omega)}+\norm{\Delta u_\alpha}_{L^2(\Omega)}\right)=C_{\textup{div}}\norm{\nabla u_\alpha}_{L^2(\Omega)}
	\end{equation}
	for all $\alpha>0$. Hence, by integration by parts, the boundary conditions in \eqref{eq:conv_1}, Cauchy-Schwarz inequality and \eqref{eq:conv_2}, we have that
	\begin{equation*}
		\int_\Omega|\nabla u_\alpha|^2\dx=\int_{\partial\Omega}u_\alpha\partial_{\nnu}u_\alpha\ds \leq \int_{\partial\Omega}f_\alpha\partial_{\nnu}u_\alpha\ds \leq C_{\textup{div}}\norm{f_\alpha}_{H^{1/2}(\partial\Omega)}\norm{\nabla u_\alpha}_{L^2(\Omega)},
	\end{equation*}
	from which we directly deduce that
	\begin{equation}\label{eq:conv_3}
		\norm{\nabla u_\alpha}_{L^2(\Omega)}\leq C_{\textup{div}}\norm{f_\alpha}_{H^{1/2}(\partial\Omega)}\quad\text{for all }\alpha>0.
	\end{equation}
	Moreover, from \Cref{lemma:coerc} and \eqref{eq:bound_th4}, we obtain that
	\begin{align*}
		\norm{u_\alpha}_{L^2(\Omega)}^2 &\leq 2\norm{u_\alpha-V_{f_\alpha}}_{L^2(\Omega)}^2+2\norm{V_{f_\alpha}}_{L^2(\Omega)}^2 \\
		&\leq 2C_{\textup{coerc}}\left(\norm{\nabla(u_\alpha-V_{f_\alpha})}_{L^2(\Omega)}^2+\norm{u_\alpha-V_{f_\alpha}}_{L^2(\partial\Omega)}^2\right)+2\norm{V_{f_\alpha}}_{L^2(\Omega)}^2 \\
		&\leq C\left( \norm{\nabla u_\alpha}_{L^2(\Omega)}^2+\norm{V_{f_\alpha}}_{H^1(\Omega)}^2+\norm{f_\alpha}_{H^{1/2}(\partial\Omega)}^2 \right),
	\end{align*}
	for all $\alpha\geq 1$. Combining the previous inequality with \eqref{eq:conv_3} and with standard estimates for the harmonic extension (see \Cref{subsec:harm_ext}), we get that
	\begin{equation*}
		\norm{u_\alpha}_{H^1(\Omega)}\leq C\norm{f_\alpha}_{H^{1/2}(\partial\Omega)}\quad\text{for all }\alpha\geq 1.
	\end{equation*}
	Hence, from weak convergence of $\{f_\alpha\}_{\alpha\geq 1}\sub H^{1/2}(\partial\Omega)$ we deduce that $\{u_\alpha\}_{\alpha\geq 1}$ is bounded in $H^1(\Omega)$. Therefore, there exists $U\in H^1(\Omega)$ such that, up to a subsequence,
	\begin{align*}
		&u_\alpha\weak U\quad\text{weakly in }H^1(\Omega), \\
		&u_\alpha\to U\quad\text{strongly in }L^2(\Omega)~\text{and }L^2(\partial\Omega),
	\end{align*}
	as $\alpha\to +\infty$. On the other hand, since $f_\alpha\weak f$ weakly in $H^{1/2}(\partial\Omega)$, then $f_\alpha\to f$ strongly in $L^2(\partial\Omega)$ and so, taking into account \eqref{eq:bound_th4}, we deduce that $U=f$ on $\partial\Omega$. Moreover, since $u_\alpha$ is harmonic in $\Omega$ we can pass to the limit into the equation and obtain that $\Delta U=0$ in $\Omega$. Therefore, by uniqueness of the harmonic extension, we have $U=V_f$. Finally, being the limit uniquely determined, Urysohn subsequence principle implies that $u_\alpha\weak V_f$ weakly in $H^1(\Omega)$ as $\alpha\to+\infty$, without passing to subsequences. The proof is thereby complete.	
\end{proof}

We finally provide a quantification of the rate of convergence of the torsion function to its limit. In this case, more regularity is needed for the datum. We refer to \Cref{subsec:harm_ext} for more details on harmonic extensions.
\begin{lemma}[Quantitative first order expansion]\label{lemma:second_order}
	Let $\Omega$ be of class $C^{1,1}$. For any $f\in H^{3/2}(\partial\Omega)$ let $V_f\in H^{2}(\Omega)$ be its harmonic extension in $\Omega$. Then there holds
	\begin{equation*}
		\norm{\alpha U_{\partial\Omega,\alpha,f}-f}^2_{L^2(\partial\Omega)}\leq \frac{\norm{\partial_{\nnu} V_f}^2_{L^2(\partial\Omega)}}{\alpha^2}\quad\text{for all }\alpha>0.
	\end{equation*}
\end{lemma}
\begin{proof}
	In this proof, we denote $U_\alpha:=U_{\partial\Omega,\alpha,f}$. We preliminarily observe that the problem
	\begin{equation}\label{eq:second_3}
		\max\left\{ 2\int_{\partial\Omega}fu\ds-\int_\Omega|\nabla u|^2\dx-\alpha\int_{\partial\Omega}u^2\ds\colon u \in H^1(\Omega) \right\},
	\end{equation}
	coincides with
	\begin{multline}\label{eq:second_1}
		\frac{1}{\alpha^3}\max\left\{-2\int_\Omega\nabla V_f\cdot\nabla w\dx-\frac{1}{\alpha}\int_\Omega|\nabla w|^2\dx-\int_{\partial\Omega}w^2\ds\colon w\in H^1(\Omega)\right\} \\
		-\frac{1}{\alpha}\left(\frac{1}{\alpha}\int_\Omega|\nabla V_f|^2\dx-\int_{\partial\Omega}f^2\ds\right)
	\end{multline}
	by simply letting $w=\alpha(\alpha u-V_f)$. Hence, since
	\begin{equation*}
		\int_\Omega \nabla V_f\cdot\nabla w\dx=\int_{\partial\Omega}w\partial_{\nnu} V_f\ds,
	\end{equation*}
	and since $U_\alpha$ is the unique function achieving \eqref{eq:second_3}, we have that $W_\alpha:=\alpha(\alpha U_\alpha-V_f)$ is the unique function achieving
	\begin{equation*}
		\widetilde{T}_\alpha:=\max\left\{ 2\int_{\partial\Omega}w\partial_{\nnu} V_f\ds-\frac{1}{\alpha}\int_\Omega|\nabla w|^2\dx-\int_{\partial\Omega}w^2\ds\colon w\in H^1(\Omega)\right\},
	\end{equation*}
	up to replacing $w$ with $-w$. Moreover, reasoning as in \Cref{lemma:minimizer}, one can easily prove that
	\begin{equation}\label{eq:second_2}
		\widetilde{T}_\alpha=\frac{1}{\alpha}\int_\Omega|\nabla W_\alpha|^2\dx+\int_{\partial\Omega}W_\alpha^2\ds.
	\end{equation}
	On the other hand, reasoning analogously as in the proof of \eqref{eq:equiv_torsional}, one can easily see that
	\begin{equation*}
		\widetilde{T}_\alpha=\max\left\{ \frac{\displaystyle \left(\int_{\partial\Omega}w\partial_{\nnu}V_f\ds\right)^2}{\displaystyle \frac{1}{\alpha}\int_\Omega|\nabla w|^2\dx+\int_{\partial\Omega}w^2\ds}\colon w\in H^1(\Omega)\setminus\{0\}\right\}.
	\end{equation*}
	Now, by applying Cauchy-Schwarz in the previous formula and taking into account \eqref{eq:second_2}, we have that
	\begin{equation*}
		\int_{\partial\Omega}W_\alpha^2\leq \widetilde{T}_\alpha\leq \norm{\partial_{\nnu}V_f}_{L^2(\partial\Omega)}^2,
	\end{equation*}
	which concludes the proof.
\end{proof}

We remark that a quantified estimate of $\norm{\alpha U_{\partial\Omega,\alpha,f}-f}_{L^2(\partial\Omega)}$ already appears in \Cref{lemma:bound}, see \eqref{eq:bound_th4}. However, the previous \Cref{lemma:second_order} provides a stronger decay (of order $\alpha^{-2}$ rather than $\alpha^{-2/3}$) in case of a more regular datum $f$. This result is indeed applied in \Cref{lemma:rho} with $f=\partial_{\nnu}\varphi$ (being $\varphi$ a Dirichlet eigenfunction), in the case in which the domain $\Omega$ is regular enough to ensure $\varphi\in H^3(\Omega)$ (i.e. of class $C^{2,1}$), so that $\partial_{\nnu}\varphi\in H^{3/2}(\partial\Omega)$.

\section{Quantitative estimates}\label{sec:quantitative}

In the present section we put the basis for the application of the \emph{Lemma on small eigenvalues}, which is the main tool for the proof of our main results. In particular, we need precise estimates for quantities related to $T_\alpha(\partial\Omega,f)$ and its minimizer, focusing on the case in which $f$ is the normal derivative of a Dirichlet eigenfunction.

We start with a simple but crucial observation: since any Dirichlet eigenspace $E(\lambda_n)$ is finite dimensional, all the norms are equivalent.

\begin{lemma}[Equivalence of norms]\label{lemma:equiv_norms}
	Let $n\in\N$. Then, there exists a constant $C_n^{\textup{eq}}>0$, depending on $d$, $\Omega$ and $n$, such that	
	\begin{equation}\label{eq:equiv_norms_th1}
		\frac{1}{C_n^{\textup{eq}}}\leq \frac{\norm{\varphi}_{L^2(\Omega)}^2}{\norm{\partial_{\nnu}\varphi}_{L^2(\partial\Omega)}^2}\leq\frac{\norm{\varphi}_{H^1(\Omega)}^2}{\norm{\partial_{\nnu}\varphi}_{L^2(\partial\Omega)}^2}\leq C_n^\textup{eq}\quad\text{for all }\varphi\in E(\lambda_n)\setminus\{0\}.
	\end{equation}
	Moreover, if $\partial_{\nnu}\varphi\in H^{1/2}(\partial\Omega)$ for all $\varphi\in E(\lambda_n)$, then
	\begin{equation}\label{eq:equiv_norms_th2}
		\frac{1}{C_n^{\textup{eq}}}\leq\frac{\norm{V_{\partial_{\nnu}\varphi}}_{L^2(\Omega)}^2}{\norm{\partial_{\nnu}\varphi}_{L^2(\partial\Omega)}^2}\leq  \frac{\norm{V_{\partial_{\nnu}\varphi}}_{H^1(\Omega)}^2}{\norm{\partial_{\nnu}\varphi}_{L^2(\partial\Omega)}^2}\leq C_n^{\textup{eq}}\quad\text{for all }\varphi\in E(\lambda_n)\setminus\{0\}
	\end{equation}
	and, if also $\partial_{\nnu}V_{\partial_{\nnu}\varphi}\in L^2(\partial\Omega)$ for all $\varphi\in E(\lambda_n)$, then
	\begin{equation}\label{eq:equiv_norms_th3}
		\frac{1}{C_n^{\textup{eq}}}\leq \frac{\norm{\partial_{\nnu}V_{\partial_{\nnu}\varphi}}_{L^2(\partial\Omega)}^2}{\norm{\partial_{\nnu}\varphi}_{L^2(\partial\Omega)}^2}\leq C_n^{\textup{eq}}\quad\text{for all }\varphi\in E(\lambda_n)\setminus\{0\},
	\end{equation}
	where $V_{\partial_{\nnu}\varphi}\in H^1(\Omega)$ denotes the harmonic extension of $\partial_{\nnu}\varphi$ in $\Omega$.
\end{lemma}
\begin{proof}
	First of all, we observe that the following functions
	\begin{align*}
		E(\lambda_n)\ni \varphi&\mapsto \partial_{\nnu}\varphi \in L^2(\partial\Omega), \\
		E(\lambda_n)\ni \varphi&\mapsto V_{\partial_{\nnu}\varphi}\in H^1(\Omega), \\
		E(\lambda_n)\ni \varphi&\mapsto \partial_{\nnu}V_{\partial_{\nnu}\varphi}\in L^2(\partial\Omega)
	\end{align*}
	are linear, when the integrals are finite. As a consequence, the following functions, acting on $E(\lambda_n)\times E(\lambda_n)$,
	\begin{align*}
		&(\varphi,\psi)\mapsto \int_{\partial\Omega}\partial_{\nnu}\varphi\partial_{\nnu}\psi\ds, \\
		&(\varphi,\psi)\mapsto \int_\Omega\nabla V_{\partial_{\nnu}\varphi}\cdot\nabla V_{\partial_{\nnu}\psi}\dx+\int_\Omega V_{\partial_{\nnu}\varphi}V_{\partial_{\nnu}\psi}\dx, \\
		&(\varphi,\psi)\mapsto \int_{\partial\Omega}\partial_{\nnu}V_{\partial_{\nnu}\varphi}\partial_{\nnu}V_{\partial_{\nnu}\psi}\ds, 
	\end{align*}
	are well-defined (in the sense that the integrals are finite) bilinear forms. The first one is well-defined in view of \eqref{ass:Omega}, while the second and the third integrals are finite by assumption. Moreover, they all are scalar product. In fact, by classical unique continuation results we have that, if $\partial_{\nnu}\varphi\equiv\varphi\equiv 0$ on $\partial\Omega$ for some $\varphi\in E(\lambda_n)$, then necessarily $\varphi\equiv 0$ in $\Omega$. Indeed, if we extend such $\varphi$ as $\varphi\equiv 0$ in $\R^d\setminus\Omega$, then one can prove (using the divergence theorem, see \Cref{subsec:dive}) that $\varphi\in H^1(\R^d)$ and that it weakly solves $-\Delta \varphi=\lambda_n\varphi$ in $\R^d$. Thus by unique continuation at interior points (see e.g. \cite{Aronszajn}), we have that $\varphi\equiv 0$ in $\R^d$. 
	 In turn, by uniqueness of the harmonic extension, $V_{\partial_{\nnu}\varphi}\equiv 0$ in $\Omega$ if and only if $\partial_{\nnu}\varphi\equiv 0$ on $\partial\Omega$. Therefore, since $E(\lambda_n)$ is finite dimensional, every norm is equivalent, thus concluding the proof.
\end{proof}

From now on, for the sake of simplicity, we denote
\begin{equation*}
	U_\alpha^\varphi:=U_{\partial\Omega,\alpha,\partial_{\nnu}\varphi},
\end{equation*}
for any Dirichlet eigenfunction $\varphi\in H^1_0(\Omega)$ and any $\alpha>0$. Moreover, we introduce the following two quantities, which play the role of remainder terms in the estimate of the eigenvalue variation, see \Cref{sec:sketch}. More precisely, for any $n\in\N$ and $\alpha>0$, we let
\begin{equation}\label{eq:def_omega}
	\omega_n(\alpha):=\sup_{\substack{\varphi\in E(\lambda_n) \\ \norm{\varphi}_{L^2(\Omega)}=1}}\norm{U_\alpha^\varphi}_{L^2(\Omega)}^2
\end{equation}
and
\begin{equation}\label{eq:def_rho}
	\rho_n(\alpha):=\sup_{\substack{\varphi,\psi\in E(\lambda_n) \\ \norm{\varphi}_{L^2(\Omega)}=\norm{\psi}_{L^2(\Omega)}=1}}\abs{\int_{\partial\Omega}\left(U_\alpha^\psi-\frac{1}{\alpha}\partial_{\nnu}\psi\right)\partial_{\nnu}\varphi\ds}.
\end{equation}
We start with the analysis of $\omega_n(\alpha)$, for which we first need the following preliminary result.
\begin{lemma}\label{lemma:L2_weak}
	For any $n\in\N$, we have
	\begin{equation*}
		\sup_{\substack{\varphi\in E(\lambda_n) \\ \norm{\varphi}_{L^2(\Omega)}=1}}\frac{\norm{U_\alpha^\varphi}_{L^2(\Omega)}^2}{T_\alpha(\partial\Omega,\partial_{\nnu}\varphi)}\to 0\quad\text{as }\alpha\to+\infty.
	\end{equation*}
\end{lemma}
\begin{proof}
	Let us assume by contradiction that, there exists $\{\phi_i\}_i\sub E(\lambda_n)$, $\{\alpha_i\}_i\sub (0,+\infty)$ and $C>0$ such that $\norm{\phi_i}_{L^2(\Omega)}=1$, $\alpha_i\to +\infty$ as $i\to\infty$ and
	\begin{equation*}
		\frac{\|U_{\alpha_i}^{\phi_i}\|_{L^2(\Omega)}^2}{T_{\alpha_i}(\partial\Omega,\partial_{\nnu}\phi_i)}\geq C\quad\text{for all }i\in\N.
	\end{equation*}
	If we denote 
	\begin{equation*}
		U_i:=\frac{U_{\alpha_i}^{\phi_i}}{\|U_{\alpha_i}^{\phi_i}\|_{L^2(\Omega)}},
	\end{equation*}
	in view also of \eqref{eq:minimizer_th1}, we have that
	\begin{equation}\label{eq:L2_weak_1}
		\int_\Omega U_i^2\dx=1\quad\text{and}\quad\int_\Omega|\nabla U_i|^2\dx+\alpha_i\int_{\partial\Omega}U_i^2\ds\leq C^{-1}\quad\text{for all }i\in\N.
	\end{equation}
	Therefore, there exists $U\in H^1(\Omega)$ such that, up to a subsequence
	\begin{align*}
		& U_i\weak U\quad\text{weakly in }H^1(\Omega), \\
		& U_i\to U\quad\text{strongly in }L^2(\Omega)~\text{and}~L^2(\partial\Omega),
	\end{align*}
	as $i\to\infty$. Moreover, in view of \eqref{eq:L2_weak_1} we have $U=0$ on $\partial\Omega$ and that $\norm{U}_{L^2(\Omega)}=1$. Finally, since $\Delta U_i=0$ in $\Omega$ for any $i\in\N$, we can pass to the limit in the equation and obtain that $\Delta U=0$ in $\Omega$ and this is a contradiction.
\end{proof}

We now state the asymptotic behavior of $\omega_n(\alpha)$ as $\alpha\to+\infty$, which turns out to depend on the regularity of $E(\lambda_n)$.
\begin{lemma}\label{lemma:L2}
	For any $n\in\N$, $\alpha\,\omega_n(\alpha)\to 0$ as $\alpha\to+\infty$. Moreover, if $\Omega$ is of class $C^{1,1}$, then $\omega_n(\alpha)\leq C_n/\alpha^2$ for all $\alpha\geq A_n$, for some $A_n\geq 1$ and $C_n>0$ depending on $d$, $\Omega$ and $n$.
\end{lemma}
\begin{proof}
	Let us first prove that $\alpha\,\omega_n(\alpha)\to 0$ as $\alpha\to+\infty$. First of all, we observe that
	\begin{equation*}
		\omega_n(\alpha)\leq\sup_{ \substack{\varphi\in E(\lambda_n) \\ \norm{\varphi}_{L^2(\Omega)}=1}} T_\alpha(\partial\Omega,\partial_{\nnu}\varphi) \sup_{\substack{\varphi\in E(\lambda_n) \\ \norm{\varphi}_{L^2(\Omega)}=1}}\frac{\norm{U_\alpha^\varphi}_{L^2(\Omega)}^2}{T_\alpha(\partial\Omega,\partial_{\nnu}\varphi)}.
	\end{equation*}
	Moreover, we can estimate the right-hand side of the previous inequality by using \eqref{eq:bound_th1} and \eqref{eq:equiv_norms_th1}, thus obtaining
	\begin{equation}\label{eq:L2_pre}
		\omega_n(\alpha)\leq\frac{C_n^{\textup{eq}}}{\alpha} \sup_{\substack{\varphi\in E(\lambda_n) \\ \norm{\varphi}_{L^2(\Omega)}=1}}\frac{\norm{U_\alpha^\varphi}_{L^2(\Omega)}^2}{T_\alpha(\partial\Omega,\partial_{\nnu}\varphi)},\quad\text{for all }\alpha>0.
	\end{equation}
	Hence, in view of \Cref{lemma:L2_weak} we conclude the proof of the first claim.
	
	Let us now assume $\Omega$ to be of class $C^{1,1}$ so that, by elliptic regularity (see e.g. \cite[Theorem 2.4.2.5]{grisvard}) we have $E(\lambda_n)\sub H^2(\Omega)$. We notice that
	\begin{equation*}
		\sup_{\substack{\varphi\in E(\lambda_n) \\ \norm{\varphi}_{L^2(\Omega)}=1}}\frac{\norm{U_\alpha^\varphi}_{L^2(\Omega)}^2}{T_\alpha(\partial\Omega,\partial_{\nnu}\varphi)}=\frac{1}{\alpha}\sup_{\substack{\varphi\in E(\lambda_n) \\ \norm{\varphi}_{L^2(\Omega)}=1}}\frac{\norm{\alpha U_\alpha^\varphi}_{L^2(\Omega)}^2}{\alpha T_\alpha(\partial\Omega,\partial_{\nnu}\varphi)}
	\end{equation*}
	for any $\alpha>0$. Hence, in view of \eqref{eq:L2_pre}, we are reduced to prove that
	\begin{equation}\label{eq:L2_1}
		\sup_{\alpha\geq A_n}\sup_{\substack{\varphi\in E(\lambda_n) \\ \norm{\varphi}_{L^2(\Omega)}=1}}\frac{\norm{\alpha U_\alpha^\varphi}_{L^2(\Omega)}^2}{\alpha T_\alpha(\partial\Omega,\partial_{\nnu}\varphi)}\leq C_n^{\textup{eq}},
	\end{equation}
	for some $A_n\geq 1$. For the sake of simplicity, in the rest of the proof we denote $T_\alpha^\varphi:=T_\alpha(\partial\Omega,\partial_{\nnu}\varphi)$ and $V^\varphi:=V_{\partial_{\nnu}\varphi}$. Now, for any $\alpha>0$ and any $\varphi\in E(\lambda_n)\setminus\{0\}$, we can estimate
	\begin{align*}
		\frac{\norm{\alpha U_\alpha^\varphi}_{L^2(\Omega)}^2}{\alpha T_\alpha^\varphi}&\leq \frac{2\norm{\alpha U_\alpha^\varphi-V^{\varphi}}_{L^2(\Omega)}^2}{\alpha T_\alpha^\varphi}+\frac{2\norm{V^\varphi}_{L^2(\Omega)}^2}{\alpha T_\alpha^\varphi} \\
		&=\frac{2\norm{\alpha U_\alpha^\varphi-V^{\varphi}}_{L^2(\Omega)}^2}{\alpha T_\alpha^\varphi}+\frac{2\norm{V^\varphi}_{L^2(\Omega)}^2\left(\norm{\partial_{\nnu}\varphi}_{L^2(\partial\Omega)}^2-\alpha T_\alpha^\varphi\right)}{\alpha T_\alpha^\varphi \norm{\partial_{\nnu}\varphi}_{L^2(\partial\Omega)}^2}+\frac{2\norm{V^\varphi}_{L^2(\Omega)}^2}{\norm{\partial_{\nnu}\varphi}_{L^2(\partial\Omega)}^2}.
	\end{align*}
	Hence, in view of \eqref{eq:equiv_norms_th2} we have
	\begin{equation*}
		\sup_{\substack{\varphi\in E(\lambda_n) \\ \norm{\varphi}_{L^2(\Omega)}=1}}\frac{\norm{\alpha U_\alpha^\varphi}_{L^2(\Omega)}^2}{\alpha T_\alpha^\varphi}\leq 2\sup_{\substack{\varphi\in E(\lambda_n) \\ \norm{\varphi}_{L^2(\Omega)}=1}}\hspace{-0.2cm}\frac{\norm{\alpha U_\alpha^\varphi-V^{\varphi}}_{L^2(\Omega)}^2}{\alpha T_\alpha^\varphi}+2C_n^{\textup{eq}}\hspace{-0.2cm}\sup_{\substack{\varphi\in E(\lambda_n) \\ \norm{\varphi}_{L^2(\Omega)}=1}}\hspace{-0.2cm}\frac{\norm{\partial_{\nnu}\varphi}_{L^2(\partial\Omega)}^2-\alpha T_\alpha^\varphi}{\alpha T_\alpha^\varphi}+2C_n^{\textup{eq}}.
	\end{equation*}
	Therefore, sufficient conditions for \eqref{eq:L2_1} are that
	\begin{equation}\label{eq:L2_2}
		\lim_{\alpha\to+\infty}\sup_{\substack{\varphi\in E(\lambda_n) \\ \norm{\varphi}_{L^2(\Omega)}=1}}\frac{\norm{\alpha U_\alpha^\varphi-V^{\varphi}}_{L^2(\Omega)}^2}{\alpha T_\alpha^\varphi}=0
	\end{equation}
	and that
	\begin{equation}\label{eq:L2_3}
		\lim_{\alpha\to+\infty}\sup_{\substack{\varphi\in E(\lambda_n) \\ \norm{\varphi}_{L^2(\Omega)}=1}}\frac{\norm{\partial_{\nnu}\varphi}_{L^2(\partial\Omega)}^2-\alpha T_\alpha^\varphi}{\alpha T_\alpha^\varphi}=0.
	\end{equation}
	For any $\alpha>0$ let us denote by $\psi_\alpha\in E(\lambda_n)$ the function attaining the supremum in \eqref{eq:L2_2}. Since $\Omega$ is of class $C^{1,1}$, then by standard elliptic regularity we have $E(\lambda_n)\sub H^2(\Omega)$ (see e.g. \cite[Theorem 2.4.2.5]{grisvard}) and  $\norm{\psi_\alpha}_{H^2(\Omega)}^2\leq C\norm{\psi_\alpha}_{H^1(\Omega)}^2=C(\lambda_n+1)$ (see e.g. \cite[Theorem 2.3.3.2]{grisvard}), for some $C>0$ and all $\alpha>0$. Hence, we have that there exists $\psi\in E(\lambda_n)$ such that, up to a subsequence
	\begin{align*}
		&\psi_\alpha\to \psi\quad\text{strongly in }H^2(\Omega), \\
		&\partial_{\nnu}\psi_\alpha\to\partial_{\nnu}\psi\quad\text{strongly in }H^{1/2}(\partial\Omega),
	\end{align*}
	as $\alpha\to+\infty$,	where the strong convergence comes from the fact that $\psi_\alpha$ varies in a finite dimensional space. Hence, by normalization $\norm{\psi}_{L^2(\Omega)}=1$. Moreover, by classical estimates for harmonic extensions (see e.g. \cite[Theorem 3.6, (i)]{BGM2022} with $s=1$), it follows that 
	\begin{equation*}
		\|V^{\psi_\alpha}-V^\psi\|_{H^1(\Omega)}\leq C \norm{{\partial_{\nnu}\psi_\alpha-\partial_{\nnu}\psi}}_{H^{1/2}(\partial\Omega)},
	\end{equation*}
	 for some $C>0$ depending on $d$ and $\Omega$. Thus, we have that
	\begin{equation*}
		V^{\psi_\alpha}\to V^\psi\quad \text{strongly in }H^1(\Omega),~\text{as }\alpha\to+\infty.
	\end{equation*}
	Moreover, from \Cref{prop:conv_torsion} and compactness of the embedding $H^1(\Omega)\hookrightarrow L^2(\Omega)$ we have that 
	\begin{align*}
		\alpha U_\alpha^{\psi_\alpha}\weak V^\psi\quad \text{weakly in }H^1(\Omega), \\
		\alpha U_\alpha^{\psi_\alpha}\to V^\psi\quad\text{strongly in }L^2(\Omega)
	\end{align*}
	as $\alpha\to+\infty$; moreover, from \eqref{eq:bound_th2} we have that $\alpha T_\alpha^{\psi_\alpha}\to \norm{\partial_{\nnu}\psi}_{L^2(\partial\Omega)}^2>0$, as $\alpha\to+\infty$. Therefore, passing to a subsequence \eqref{eq:L2_2} holds, but, since the limit is $0$ and so it does not depend on the chosen subsequence, by Urysohn subsequence principle applied to the function
	\begin{equation*}
		\alpha\mapsto \sup_{\substack{\varphi\in E(\lambda_n) \\ \norm{\varphi}_{L^2(\Omega)}=1}}\frac{\norm{\alpha U_\alpha^\varphi-V^{\varphi}}_{L^2(\Omega)}^2}{\alpha T_\alpha^\varphi}
	\end{equation*}
	we have that \eqref{eq:L2_2} holds true. Using \eqref{eq:bound_th2} in place of \Cref{prop:conv_torsion}, one can analogously prove \eqref{eq:L2_3}. The proof is thereby complete.	 
\end{proof}

In the following, we state the asymptotic behavior of $\rho_n(\alpha)$ as $\alpha\to+\infty$, which again depends on the regularity of $E(\lambda_n)$.
\begin{lemma}\label{lemma:rho}
	For any $n\in\N$, $\alpha\rho_n(\alpha)\to 0$ as $\alpha\to+\infty$. Moreover,
	\begin{itemize}
		\item if $\Omega$ is of class $C^{1,1}$, then $\rho_n(\alpha)\leq C_n/\alpha^{4/3}$
		\item if $\Omega$ is of class $C^{2,1}$, then $\rho_n(\alpha)\leq C_n/\alpha^2$
	\end{itemize}
	for all $\alpha\geq A_n$, for some $A_n\geq 1$ and $C_n>0$ depending on $d$, $\Omega$ and $n$.
\end{lemma}
\begin{proof}
	Let us first prove that $\alpha\rho_n(\alpha)\to 0$ as $\alpha\to+\infty$. One can easily see that the supremum defining $\rho_n(\alpha)$ is achieved, hence we denote by $\psi_\alpha,\varphi_\alpha\in E(\lambda_n)$ the maximizers. Since
	\begin{equation*}
		\norm{\psi_\alpha}_{H^1(\Omega)}^2=\norm{\varphi_\alpha}_{H^1(\Omega)}^2=\lambda_n+1,
	\end{equation*}
	we have that there exists $\psi,\varphi\in E(\lambda_n)$ such that, up to a subsequence
	\begin{align*}
		&\psi_\alpha\to \psi\quad\text{strongly in }H^1(\Omega), \\
		&\varphi_\alpha\to \varphi\quad\text{strongly in }H^1(\Omega),
	\end{align*}
	as $\alpha\to+\infty$. From this and integration by parts, we can deduce that
	\begin{equation*}
		\begin{aligned}
			&\partial_{\nnu}\psi_\alpha\weak \partial_{\nnu}\psi \quad\text{weakly in }L^2(\partial\Omega), \\
			&\partial_{\nnu}\varphi_\alpha\weak \partial_{\nnu}\varphi \quad\text{weakly in }L^2(\partial\Omega),
		\end{aligned}
	\end{equation*}
	as $\alpha\to+\infty$ and, since such functions vary in a finite dimensional space, the convergence is actually strong. Therefore, since in view of \Cref{lemma:weak} we have $\alpha U_\alpha^{\psi_\alpha}\weak \partial_{\nnu}\psi$ weakly in $L^2(\partial\Omega)$, as $\alpha\to +\infty$, then there holds
	\begin{equation*}
		\alpha \rho_n(\alpha)=\abs{\int_{\partial\Omega}(\alpha U_\alpha^{\psi_\alpha}-\partial_{\nnu}\psi_\alpha)\partial_{\nnu}\varphi_\alpha\ds}\to 0,
	\end{equation*}
	as $\alpha\to+\infty$, thus proving the first part of the statement. If $\Omega$ is of class $C^{1,1}$, then $E(\lambda_n)\sub H^2(\Omega)$ and so $\partial_{\nnu}\varphi \in H^{1/2}(\partial\Omega)$ for all $\varphi\in E(\lambda_n)$. Hence, in view of Cauchy-Schwarz inequality and \eqref{eq:bound_th4} we have
	\begin{align*}
		\abs{\int_{\partial\Omega}(\alpha U_\alpha^\psi-\partial_{\nnu}\psi)\partial_{\nnu}\varphi\ds}&\leq \lVert\alpha U_\alpha^\psi-\partial_{\nnu}\psi\rVert_{L^2(\partial\Omega)}\norm{\partial_{\nnu}\varphi}_{L^2(\partial\Omega)} \\
		&\leq\frac{2\norm{\partial_{\nnu}\psi}_{L^2(\partial\Omega)}^{1/3}\norm{\nabla V_{\partial_{\nnu}\psi}}_{L^2(\Omega)}^{2/3}\norm{\partial_{\nnu}\varphi}_{L^2(\partial\Omega)}}{\alpha^{1/3}},
	\end{align*}
	for all $\alpha>0$ and all $\varphi,\psi\in E(\lambda_n)$ with unitary $L^2(\Omega)$-norm, where $V_{\partial_{\nnu}\psi}\in H^1(\Omega)$ denotes the harmonic extension of $\partial_{\nnu}\psi$ in $\Omega$. Then, in view of \Cref{lemma:equiv_norms} (in particular, \eqref{eq:equiv_norms_th1} and \eqref{eq:equiv_norms_th2}) we conclude the proof of the second part. Finally, if $\Omega$ is of class $C^{2,1}$, then $E(\lambda_n)\sub H^3(\Omega)$ (see e.g. \cite[Theorem 2.5.1.1]{grisvard}), so $\partial_{\nnu}\varphi\in H^{3/2}(\partial\Omega)$ and $V_{\partial_{\nnu}\varphi}\in H^2(\Omega)$, for all $\varphi\in E(\lambda_n)$. Hence, we apply \Cref{lemma:second_order} in place of \eqref{eq:bound_th4} and reason analogously to the previous step, thus reaching the conclusion of the proof.
\end{proof}

We now define some objects which will play a role in the proof of our main results and, in particular, in the application of \emph{Lemma on small eigenvalues} (\Cref{lemma:CdV}), see \Cref{sec:sketch}. Namely, for any $\alpha>0$ and any $n\in\N$, we consider the bilinear form $q_\alpha^n\colon H^1(\Omega)\times H^1(\Omega) \to \R$ introduced in \eqref{eq:q_intr}, which we recall to be defined as 
\begin{equation}\label{eq:q_alpha}
	q_\alpha^n(u,v):=\int_\Omega\nabla u\cdot\nabla v\dx+\alpha\int_{\partial\Omega}uv\ds-\lambda_n\int_\Omega uv\dx.
\end{equation}
Moreover, we consider the following linear operator
\begin{equation}\label{eq:P_alpha}
	\begin{aligned}
	\mathscr{P}_\alpha^n\colon E(\lambda_n)&\to L^2(\Omega), \\
	\varphi&\mapsto \mathscr{P}_\alpha^n(\varphi):= \varphi-U_\alpha^\varphi,
\end{aligned}
\end{equation}
which then satisfies
\begin{equation}\label{eq:norm_P}
	\norm{\mathscr{P}_\alpha^n-I}_{\mathcal{L}(E(\lambda_n),L^2(\Omega))}^2=\omega_n(\alpha),
\end{equation}
with $\omega_n(\alpha)$ being as in \eqref{eq:def_omega}. This allows to define the approximating eigenspace $F$ as
\begin{equation}\label{eq:F_alpha}
	F_\alpha^n:=\mathscr{P}_\alpha^n(E(\lambda_n)).
\end{equation}
In the following, we investigate the behavior of the quadratic form $q_\alpha^n$ restricted to $F_\alpha^n$, which again is a crucial step in the application of \Cref{lemma:CdV}
\begin{lemma}\label{lemma:Q_phi_psi}
	There exists $C_n>0$ and $A_n\geq 1$ such that
	\begin{equation*}
		\sigma_n(\alpha):=\sup_{\substack{\varphi,\psi\in E(\lambda_n) \\ \norm{\varphi}_{L^2(\Omega)}=\norm{\psi}_{L^2(\Omega)}=1}}\abs{q_\alpha^n(\mathscr{P}_\alpha^n(\varphi),\mathscr{P}_\alpha^n(\psi))+\alpha^{-1}\int_{\partial\Omega}\partial_{\nnu}\varphi\,\partial_{\nnu}\psi\ds}\leq C_n\left(\rho_n(\alpha)+\omega_n(\alpha)\right),
	\end{equation*}
	for all $\alpha\geq A_n$. In particular, $\alpha\sigma_n(\alpha)\to 0$, as $\alpha\to+\infty$. Moreover,
	\begin{itemize}
		\item if $\Omega$ is of class $C^{1,1}$, then $\sigma_n(\alpha)\leq C_n/\alpha^{4/3}$ for all $\alpha\geq A_n$;
		\item if $\Omega$ is of class $C^{2,1}$, then $\sigma_n(\alpha)\leq C_n/\alpha^2$ for all $\alpha\geq A_n$.
	\end{itemize}
\end{lemma}
\begin{proof}
	For sake of simplicity, in this proof we denote $U_\alpha^\varphi:=U_{\partial\Omega,\alpha,\partial_{\nnu}\varphi}$ and $U_\alpha^\psi:=U_{\partial\Omega,\alpha,\partial_{\nnu}\psi}$.	We first expand
	\begin{equation}\label{eq:Q_1}
		q_\alpha^n(\varphi-U_\alpha^\varphi,\psi-U_\alpha^\psi)=q_\alpha^n(\varphi,\psi)-q_\alpha^n(\varphi,U_\alpha^\psi)-q_\alpha^n(\psi,U_\alpha^\varphi)+q_\alpha^n(U_\alpha^\varphi,U_\alpha^\psi)
	\end{equation}
	and we separately study each term. Since $\varphi,\psi\in E(\lambda_n)$, obviously $q_\alpha^n(\varphi,\psi)=0$. Then, taking into account that $\varphi=0$ on $\partial\Omega$ and integrating by parts we have that
	\begin{equation}\label{eq:Q_2}
		q_\alpha^n(\varphi,U_\alpha^\psi)=\int_\Omega\nabla \varphi\cdot\nabla U_\alpha^\psi\dx-\lambda_n\int_\Omega\varphi U_\alpha^\psi\dx=\int_{\partial\Omega}U_\alpha^\psi\partial_{\nnu}\varphi\ds.
	\end{equation}
	Analogously, we have
	\begin{equation}\label{eq:Q_3}
		q_\alpha^n(\psi,U_\alpha^\varphi)=\int_\Omega\nabla \psi\cdot\nabla U_\alpha^\varphi\dx-\lambda_n\int_\Omega\psi U_\alpha^\varphi\dx=\int_{\partial\Omega}U_\alpha^\varphi\partial_{\nnu}\psi\ds.
	\end{equation}
	Finally, thanks to the equation satisfied by $U_\alpha^\psi$ (see \eqref{eq:minimizer_th2}), we have that
	\begin{equation}\label{eq:Q_4}
		\begin{aligned}
			q_\alpha^n(U_\alpha^\varphi,U_\alpha^\psi)&=\int_\Omega\nabla U_\alpha^\varphi\cdot \nabla U_\alpha^\psi\dx+\alpha\int_{\partial\Omega}U_\alpha^\varphi U_\alpha^\psi\ds-\lambda_n\int_\Omega U_\alpha^\varphi U_\alpha^\psi\dx \\
			&=\int_{\partial\Omega}U_\alpha^\varphi\partial_{\nnu}\psi\ds-\lambda_n\int_\Omega U_\alpha^\varphi U_\alpha^\psi\dx.
		\end{aligned}
	\end{equation}
	Now, plugging \eqref{eq:Q_2}, \eqref{eq:Q_3} and \eqref{eq:Q_4} into \eqref{eq:Q_1}, we obtain that
	\begin{equation*}
		q_\alpha^n(\varphi-U_\alpha^\varphi,\psi-U_\alpha^\psi)=-\int_{\partial\Omega}U_\alpha^\psi\partial_{\nnu}\varphi\ds-\lambda_n\int_\Omega U_\alpha^\varphi U_\alpha^\psi\dx.
	\end{equation*}
	Therefore, in view of the definition of $\rho_n$ and $\omega_n$, and Cauchy-Schwarz inequality, we have
	\begin{align*}
		&\sup_{\substack{\varphi,\psi\in E(\lambda_n) \\ \norm{\varphi}_{L^2(\Omega)}=\norm{\psi}_{L^2(\Omega)}=1}}\abs{q_\alpha^n(\varphi-U_\alpha^\varphi,\psi-U_\alpha^\psi)+\alpha^{-1}\int_{\partial\Omega}\partial_{\nnu}\varphi\partial_{\nnu}\psi\ds} \\
		&\leq\sup_{\substack{\varphi,\psi\in E(\lambda_n) \\ \norm{\varphi}_{L^2(\Omega)}=\norm{\psi}_{L^2(\Omega)}=1}}\abs{\int_{\partial\Omega}( U_\alpha^\psi-\alpha^{-1}\partial_{\nnu}\psi)\partial_{\nnu}\varphi\ds}+\lambda_n\sup_{\substack{\varphi,\psi\in E(\lambda_n) \\ \norm{\varphi}_{L^2(\Omega)}=\norm{\psi}_{L^2(\Omega)}=1}}\abs{\int_{\Omega}U_\alpha^\varphi U_\alpha^\psi\dx} \\
		&\leq \rho_n(\alpha)+\lambda_n\omega_n(\alpha).
	\end{align*}
	Then, by \Cref{lemma:L2} and \Cref{lemma:rho} we can conclude the proof.	
\end{proof}

The following is a straightforward consequence of \Cref{lemma:L2}, \Cref{lemma:rho} and \Cref{lemma:Q_phi_psi}.
\begin{corollary}\label{cor:eigen_q}
	Let $n\in\N$, let $F_\alpha^n$ be as in \eqref{eq:F_alpha} and let $q_\alpha^n$ be as in \eqref{eq:q_alpha}. Let $\{\xi_{n,i}^\alpha\}_{i=1,\dots,m}$ be the set of eigenvalues of $q_\alpha^n$ restricted to $F_\alpha^n$ and let $\{\mu_{n,i}\}_{i=1,\dots,m}$ be the set of eigenvalues (in descending order) of the scalar product \eqref{eq:intr_bilinear} defined on $E(\lambda_n)$, i.e.
	\begin{equation*}
		\mu_{n,i}:=\min_{\substack{G\sub E(\lambda_n) \\ \dim G=m-i+1} }\max_{\substack{u\in G \\ u\neq 0}}\,\frac{\norm{\partial_{\nnu}u}_{L^2(\partial\Omega)}^2}{\displaystyle \norm{u}_{L^2(\Omega)}^2}.
	\end{equation*}
	In particular, in view of the choice of the eigenbasis as in \eqref{eq:hp_eigen}, we have
	\begin{equation*}
		\mu_{n,i}=\int_{\partial\Omega}(\partial_{\nnu}\varphi_{n+i-1})^2\ds.
	\end{equation*}
	Then there exists $C_n>0$ and $A_n\geq 1$ (depending on $d$, $\Omega$ and $n$) such that
	\begin{equation*}
		\abs{\xi_{n,i}^\alpha+\alpha^{-1}\mu_{n,i}}\leq 
		C_n(\rho_n(\alpha)+\omega_n(\alpha))\quad\text{for all }\alpha\geq A_n,
	\end{equation*}
	where $\rho_n$ is as in \eqref{eq:def_rho} and $\omega_n$ as in \eqref{eq:def_omega}. Moreover,
	\begin{itemize}
		\item if $\Omega$ is of class $C^{1,1}$, then 
		\begin{equation*}
			\abs{\xi_{n,i}^\alpha+\alpha^{-1}\mu_{n,i}}\leq \frac{C_n}{\alpha^{4/3}},
		\end{equation*}
		\item if $\Omega$ is of class $C^{2,1}$, then
		\begin{equation*}
			\abs{\xi_{n,i}^\alpha+\alpha^{-1}\mu_{n,i}}\leq \frac{C_n}{\alpha^2},
		\end{equation*}
		for all $\alpha\geq A_n$.
	\end{itemize}
\end{corollary}

\section{Proof of the main results}	

It is now time for the proof of our main results, which essentially follows from the application of the \emph{Lemma on small eigenvalues}, in view of the estimates derived in \Cref{sec:quantitative}. Let us recall some notation. The main characters of the present section are:
\begin{itemize}
	\item the Dirichlet eigenvalue $\lambda_n$, with eigenspace $E(\lambda_n)$ and multiplicity $m:=\dim E(\lambda_n)$;
	\item the Robin eigenvalues $\{\lambda_{n+i-1}^\alpha\}_{i=1,\dots,m}$, which are converging to $\lambda_n$ as $\alpha\to+\infty$, the corresponding eigenspaces $E(\lambda_{n+i-1}^\alpha)$, for $i=1,\dots,m$, and their sum $\mathcal{E}_\alpha^n$ (see \eqref{eq:E_alpha});
	\item the spectral projection $\Pi_\alpha^n\colon L^2(\Omega)\to \mathcal{E}_\alpha^n$ with respect to $L^2(\Omega)$;
	\item the bilinear form $q_\alpha^n$, defined in \eqref{eq:q_alpha};
	\item the approximating eigenspace $F_\alpha^n$, defined in \eqref{eq:F_alpha};
	\item the eigenvalues $\{\xi_{n,i}^\alpha\}_{i=1,\dots,m}$ of $q_\alpha^n$ restricted to $F_\alpha^n$;
	\item the operator $\mathscr{P}_\alpha^n$, defined in \eqref{eq:P_alpha};
	\item the remainder terms $\omega_n(\alpha)$ and $\rho_n(\alpha)$, defined in \eqref{eq:def_omega} and \eqref{eq:def_rho}, respectively;
\end{itemize}

The following result basically contains the application of \Cref{lemma:CdV} to our framework. We sketched it in \Cref{sec:sketch}.
\begin{proposition}\label{prop:appl_abs_lemma}
	There exist constants $C_n>0$ and $A_n\geq 1$ (depending on $d$, $\Omega$ and $n$) such that
	\begin{equation}\label{eq:app_abs_lemma_th1}
		\abs{\lambda_{n+i-1}^\alpha-\lambda_n-\xi_{n,i}^\alpha}\leq C_n\omega_n(\alpha)
	\end{equation}
	for all $\alpha\geq A_n$ and all $i=1,\dots,m$ and that
	\begin{equation}\label{eq:app_abs_lemma_th2}
		\frac{\norm{\mathscr{P}_\alpha^n(\varphi)-\Pi_\alpha^n\mathscr{P}_\alpha^n(\varphi)}_{L^2(\Omega)}^2}{\norm{\mathscr{P}_\alpha^n(\varphi)}_{L^2(\Omega)}^2}\leq C_n\omega_n(\alpha)
	\end{equation}
	for all $\alpha\geq A_n$ and all $\varphi\in E(\lambda_n)\setminus\{0\}$.
\end{proposition}
\begin{proof}
	We make use of \Cref{lemma:CdV} by choosing
	\begin{equation*}
		\mathcal{H}:=L^2(\Omega),\quad \mathcal{D}:=H^1(\Omega),\quad F:=F_\alpha^n,\quad q:=q_\alpha^n,
	\end{equation*}
	with $F_\alpha^n$ and $q_\alpha^n$ being as in \eqref{eq:F_alpha} and \eqref{eq:q_alpha}, respectively. Let us now check the validity of the assumptions of \Cref{lemma:CdV}. Being $\{\nu_i:=\lambda_i^\alpha-\lambda_n\}_{i\in\N}$ the sequence of eigenvalues of $q_\alpha^n$, in view of \eqref{eq:convergence_eigenvalues}, we have that assumption (H1) in \Cref{lemma:CdV} holds with
	\begin{equation*}
		\gamma:=\frac{1}{2}\min\left\{\lambda_n-\lambda_{n-1},\lambda_{n+m}-\lambda_n \right\}
	\end{equation*}
	when $n\geq 2$ and 
	\begin{equation*}
		\gamma:=\frac{1}{2}(\lambda_2-\lambda_1)
	\end{equation*}
	when $n=1$, by choosing $\alpha$ sufficiently large (depending on $d$, $\Omega$ and $n$). Concerning (H2), since $\alpha\omega_n(\alpha)\to 0$ as $\alpha\to+\infty$, we can assume $\omega_n(\alpha)\leq 1/2$ for $\alpha$ sufficiently large and we observe that, in view of \eqref{eq:norm_P} $\mathscr{P}_\alpha^n$ is a bijection for $\alpha$ sufficiently large. Hence, for any $v\in F$ there exists a unique $\varphi\in E(\lambda_n)$ such that $v=\mathscr{P}_\alpha^n(\varphi)=\varphi-U_\alpha^\varphi$. Now, by integrating by parts and using the equations satisfied by $\varphi$ and $U_\alpha^\varphi$, we obtain that
	\begin{align*}
		q_\alpha^n(u,v)&=\int_\Omega\nabla u\cdot\nabla v\dx+\alpha\int_{\partial\Omega}uv\ds-\lambda_n\int_\Omega uv\dx \\
		&=\int_\Omega \nabla u\cdot\nabla(\varphi-U_\alpha^\varphi)\dx+\alpha\int_{\partial\Omega}u(\varphi -U_\alpha^\varphi)\ds-\lambda_n\int_{\Omega}u(\varphi-U_\alpha^\varphi)\dx \\
		&=\lambda_n\int_\Omega u U_\alpha^\varphi\dx.
	\end{align*}
	Therefore,
	\begin{equation*}
		\abs{q_\alpha^n(u,v)}\leq \lambda_n\norm{U_\alpha^\varphi}_{L^2(\Omega)}\norm{u}_{L^2(\Omega)}=\frac{\lambda_n\norm{U_\alpha^\varphi}_{L^2(\Omega)}}{\norm{\varphi-U_\alpha^\varphi}_{L^2(\Omega)}}\norm{v}_{L^2(\Omega)}\norm{u}_{L^2(\Omega)},
	\end{equation*}
	if $v=\varphi-U_\alpha^\varphi\neq 0$, but $v=0$ if and only if $\varphi=0$, being $\mathscr{P}_\alpha^n$ a bijection. Hence,
	\begin{equation*}
		\delta=\delta_{n,\alpha}\leq\sup_{ \substack{\varphi\in E(\lambda_n) \\ \norm{\varphi}_{L^2(\Omega)}=1}} \frac{\lambda_n\norm{U_\alpha^\varphi}_{L^2(\Omega)}}{\norm{\varphi-U_\alpha^\varphi}_{L^2(\Omega)}}.
	\end{equation*}
	We now aim at providing an estimate on the behavior of $\delta_{n,\alpha}$. Namely, we have that
	\begin{align*}
		\delta_{n,\alpha}\leq\sup_{ \substack{\varphi\in E(\lambda_n) \\ \norm{\varphi}_{L^2(\Omega)}=1}} \frac{\lambda_n\norm{(\mathscr{P}_\alpha^n-I)(\varphi)}_{L^2(\Omega)}}{\norm{\mathscr{P}_\alpha^n(\varphi)}_{L^2(\Omega)}}&\leq \sup_{ \substack{\varphi\in E(\lambda_n) \\ \norm{\varphi}_{L^2(\Omega)}=1}} \frac{\lambda_n\norm{(\mathscr{P}_\alpha^n-I)(\varphi)}_{L^2(\Omega)}}{\norm{\varphi}_{L^2(\Omega)}-\norm{(\mathscr{P}_\alpha^n-I)(\varphi)}_{L^2(\Omega)}} \\
		&\leq \frac{\lambda_n\norm{\mathscr{P}_\alpha^n-I}_{\mathcal{L}(E(\lambda_n),L^2(\Omega))}}{1-\norm{\mathscr{P}_\alpha^n-I}_{\mathcal{L}(E(\lambda_n),L^2(\Omega))}}=\frac{\lambda_n\sqrt{\omega_n(\alpha)}}{1-\sqrt{\omega_n(\alpha)}},
	\end{align*}
	where, in the last line, we used \eqref{eq:norm_P}. Therefore, by assuming $\alpha$ sufficiently large (depending on $d$, $\Omega$ and $n$), we have that also (H2) is satisfied. At this point, thanks to \Cref{lemma:CdV} we can conclude the proof.
\end{proof}

We are now ready to prove our main results \Cref{thm:main} and \Cref{thm:eigenfunctions}.

\begin{proof}[Proof of \Cref{thm:main}]
	By combining \eqref{eq:app_abs_lemma_th1} with \Cref{cor:eigen_q} we have that
	\begin{align}
		\abs{\lambda_n-\lambda_{n+i-1}^\alpha-\alpha^{-1}\mu_{n,i}}&\leq \abs{\lambda_{n+i-1}^\alpha-\lambda_n-\xi_{n,i}^\alpha}+\abs{\xi_{n,i}^\alpha+\alpha^{-1}\mu_{n,i}} \notag \\
		&\leq C_n(\omega_n(\alpha)+\rho_n(\alpha)), \label{eq:main_1}
	\end{align}
	for some $C_n>0$ and $\alpha$ sufficiently large (depending on $d$, $\Omega$ and $n$). Thanks to \Cref{lemma:L2} and \Cref{lemma:rho}, we have that both $\alpha\,\rho_n(\alpha)$ and $\alpha\,\omega_n(\alpha)$ vanish as $\alpha\to+\infty$, thus concluding the proof of the first part. We point out that we assumed that the eigenbasis of $E(\lambda_n)$ diagonalizes the scalar product \eqref{eq:intr_bilinear}. The remaining estimates can be proved analogously, by exploiting the explicit rate of convergence available, in case of higher regularity, in \Cref{lemma:L2} and \Cref{lemma:rho}.
\end{proof}

\begin{proof}[Proof of \Cref{thm:eigenfunctions}]
	In the whole proof, we denote by $C>0$ a constant depending only on $d$, $\Omega$ and $n$ (the index of the eigenvalue), which may change value from line to line. Analogously, by saying \enquote{for $\alpha$ sufficiently large}, we mean a dependence of $\alpha$ on $d$, $\Omega$ and $n$.
	
	The idea is to start from \eqref{eq:app_abs_lemma_th2} and refine it in various steps. Let us fix $\varphi\in E(\lambda_n)\setminus\{0\}$ and let us denote $h_\alpha^\varphi:=\mathscr{P}_\alpha^n(\varphi)$.
	
	\noindent\textbf{Step 1:} we claim that
	\begin{equation}\label{eq:eigenfunctions_claim1}
		\norm{h_\alpha^\varphi-\Pi_\alpha^n h_\alpha^\varphi}_{\mathcal{H}_\alpha}^2\leq C\norm{\varphi}_{L^2(\Omega)}^2\left(\omega_n(\alpha)+\frac{\sqrt{\omega_n(\alpha)}}{\alpha}\right)\quad\text{for $\alpha$ sufficiently large.}
	\end{equation}
	From \eqref{eq:app_abs_lemma_th2} and \eqref{eq:norm_P} we have that
	\begin{align}
		\norm{h_\alpha^\varphi-\Pi_\alpha^n h_\alpha^\varphi}_{L^2(\Omega)}^2 &\leq C\omega_n(\alpha)\left(\norm{\varphi}_{L^2(\Omega)}^2+\norm{(\mathscr{P}_\alpha^n-I)(\varphi)}_{L^2(\Omega)}^2\right) \notag\\
		&\leq C\omega_n(\alpha)\norm{\varphi}_{L^2(\Omega)}^2.\label{eq:eigenfunctions_1}
	\end{align}
	Since we aim at estimate the $L^2(\Omega)$-norm of the gradient, we consider the equation satisfied by $h_\alpha^\varphi-\Pi_\alpha^n h_\alpha^\varphi$, which one can easily check to be
	\begin{equation*}
		\begin{bvp}
			-\Delta(h_\alpha^\varphi-\Pi_\alpha^nh_\alpha^\varphi)&=\lambda_n (h_\alpha^\varphi-\Pi_\alpha^n h_\alpha^\varphi)+\lambda_n U_\alpha^\varphi+\lambda_n \Pi_\alpha^n h_\alpha^\varphi+\Delta(\Pi_\alpha^n h_\alpha^n), &&\text{in }\Omega, \\
			\partial_{\nnu}(h_\alpha^\varphi-\Pi_\alpha^n h_\alpha^\varphi)+\alpha(h_\alpha^\varphi-\Pi_\alpha^n h_\alpha^\varphi)&=0, &&\text{on }\partial\Omega,
		\end{bvp}
	\end{equation*}
	intended in a weak sense. If we now multiply by $h_\alpha^\varphi-\Pi_\alpha^n h_\alpha^\varphi$ and integrate by parts, we obtain that
	\begin{multline}\label{eq:eigenfunctions_2}
		\norm{h_\alpha^\varphi-\Pi_\alpha^n h_\alpha^\varphi}_{\mathcal{H}_\alpha}^2 \\
		=\int_{\Omega}\Big(\lambda_n(h_\alpha^\varphi-\Pi_\alpha^n h_\alpha^\varphi)^2+\lambda_n U_\alpha^\varphi(h_\alpha^\varphi-\Pi_\alpha^n h_\alpha^\varphi)+(\lambda_n\Pi_\alpha^n h_\alpha+\Delta(\Pi_\alpha^n h_\alpha^\varphi))(h_\alpha^\varphi-\Pi_\alpha^n h_\alpha^\varphi)\Big)
	\end{multline}
	and we now analyze each term of the right-hand side of the previous identity. While we can estimate the first term just by using \eqref{eq:eigenfunctions_1}, for the second term we use Cauchy-Schwarz inequality, the definition on $\omega_n(\alpha)$ and \eqref{eq:eigenfunctions_1}, thus obtaining that
	\begin{equation}\label{eq:eigenfunctions_5}
		\lambda_n\abs{\int_\Omega U_\alpha^\varphi(h_\alpha^\varphi-\Pi_\alpha^n h_\alpha^\varphi)\dx}\leq \lambda_n\norm{U_\alpha^\varphi}_{L^2(\Omega)}\norm{h_\alpha^\varphi-\Pi_\alpha^n h_\alpha^\varphi}_{L^2(\Omega)}\leq C\omega_n(\alpha)\norm{\varphi}_{L^2(\Omega)}^2.
	\end{equation}
	Finally, for what concerns the last term in \eqref{eq:eigenfunctions_2}, we write
	\begin{equation*}
		\Pi_\alpha^n h_\alpha^\varphi=\sum_{i=1}^m a_i^\alpha\varphi_{n+i-1}^\alpha,
	\end{equation*}
	for some $a_i^\alpha\in\R$, where $\{\varphi_{n+i-1}^\alpha\}_{i=1,\dots,m}$ denotes an $L^2(\Omega)$-orthonormal eigenbasis of $\mathcal{E}_\alpha^n$. This way, we have
	\begin{equation*}
		-\Delta(\Pi_\alpha^n h_\alpha^\varphi)=\sum_{i=1}^m a_i^\alpha\lambda_{n+i-1}^\alpha\varphi_{n+i-1}^\alpha,\quad\text{in }\Omega
	\end{equation*}
	and so
	\begin{align}
		\norm{\lambda_n \Pi_\alpha^n h_\alpha^\varphi+\Delta(\Pi_\alpha^n h_\alpha^\varphi)}_{L^2(\Omega)}&=\norm{\sum_{i=1}^m a_i^\alpha(\lambda_n-\lambda_{n+i-1}^\alpha)\varphi_{n+i-1}^\alpha}_{L^2(\Omega)} \notag \\
		&\leq m\norm{\Pi_\alpha^n h_\alpha^\varphi}_{L^2(\Omega)}\max_{i\in\{1,\dots,m\}}|\lambda_n-\lambda_{n+i-1}^\alpha| \notag\\
		&\leq m\norm{h_\alpha^\varphi}_{L^2(\Omega)}\max_{i\in\{1,\dots,m\}}|\lambda_n-\lambda_{n+i-1}^\alpha|\notag \\
		&\leq \frac{3}{2}m\norm{\varphi}_{L^2(\Omega)}\max_{i\in\{1,\dots,m\}}|\lambda_n-\lambda_{n+i-1}^\alpha|,\label{eq:eigenfunctions_3}
	\end{align}
	where in the last two lines we used the fact that $\Pi_\alpha^n$ is a projection and, assuming $\alpha$ sufficiently large, that
	\begin{equation*}
		\norm{h_\alpha^\varphi}_{L^2(\Omega)}\leq \frac{3}{2}\norm{\varphi}_{L^2(\Omega)}.
	\end{equation*}
	 Moreover, by \eqref{eq:main_1} we have that
	\begin{equation*}
			\max_{i\in\{1,\dots,m\}}\abs{\lambda_n-\lambda_{n+i-1}^\alpha}\leq C(\rho_n(\alpha)+\omega_n(\alpha))+\alpha^{-1}\max_{i\in\{1,\dots,m\}}\int_\Omega(\partial_{\nnu}\varphi_{n+i-1})^2\ds,
	\end{equation*}
	and, since $\alpha\,\omega_n(\alpha)$ and $\alpha\,\rho_n(\alpha)$ vanish as $\alpha\to+\infty$, this implies that
	\begin{equation}
			\max_{i\in\{1,\dots,m\}}\abs{\lambda_n-\lambda_{n+i-1}^\alpha}\leq\frac{C}{\alpha},\label{eq:eigenfunctions_4}
	\end{equation}
	for $\alpha$ sufficiently large. Now, combining \eqref{eq:eigenfunctions_3}, \eqref{eq:eigenfunctions_4} and \eqref{eq:eigenfunctions_1} we deduce that
	\begin{align}
		\abs{\int_\Omega(\lambda_n \Pi_\alpha^n h_\alpha^\varphi+\Delta(\Pi_\alpha^n h_\alpha^\varphi))(h_\alpha^\varphi-\Pi_\alpha^n h_\alpha^\varphi)\dx} &\leq \norm{\lambda_n \Pi_\alpha^n h_\alpha^\varphi+\Delta(\Pi_\alpha^n h_\alpha^\varphi)}_{L^2(\Omega)} \norm{h_\alpha^\varphi-\Pi_\alpha^n h_\alpha^\varphi}_{L^2(\Omega)}\notag \\
		&\leq C\frac{\sqrt{\omega_n(\alpha)}}{\alpha}\norm{\varphi}_{L^2(\Omega)}^2 \label{eq:eigenfunctions_6}
	\end{align}
	If we now put together \eqref{eq:eigenfunctions_2} with \eqref{eq:eigenfunctions_1}, \eqref{eq:eigenfunctions_5} and \eqref{eq:eigenfunctions_6}, we obtain \eqref{eq:eigenfunctions_claim1}.
	
	\noindent\textbf{Step 2:} we claim that
	\begin{equation}\label{eq:eigenfunctions_claim2}
		\norm{h_\alpha^\varphi-\Pi_\alpha^n \varphi}_{\mathcal{H}_\alpha}^2\leq C\norm{\varphi}_{L^2(\Omega)}^2\left(\omega_n(\alpha)+\frac{\sqrt{\omega_n(\alpha)}}{\alpha}\right)\quad\text{for $\alpha$ sufficiently large.}
	\end{equation}
	In order to see this, let us first write
	\begin{equation*}
		\Pi_\alpha^n h_\alpha^\varphi=\sum_{i=1}^ma_i^\alpha \varphi_{n+i-1}^\alpha\quad\text{and}\quad \Pi_\alpha^n\varphi=\sum_{i=1}^mb_i^\alpha\varphi_{n+i-1}^\alpha,
	\end{equation*}
	for some $a_i^\alpha,b_i^\alpha\in\R$, so that, in view of the boundary conditions, we have
	\begin{equation*}
		\norm{\Pi_\alpha^n h_\alpha^\varphi-\Pi_\alpha^n\varphi}_{\mathcal{H}_\alpha}^2=\sum_{i=1}^m\lambda_{n+i-1}^\alpha(a_i^\alpha-b_i^\alpha)^2.
	\end{equation*}
	Moreover, by monotonicity of eigenvalues and since $\Pi_\alpha^n$ is a linear projection, we have
	\begin{equation*}
		\norm{\Pi_\alpha^n h_\alpha^\varphi-\Pi_\alpha^n\varphi}_{\mathcal{H}_\alpha}^2\leq \lambda_{n+m-1}^\alpha\norm{\Pi_\alpha^n h_\alpha^\varphi-\Pi_\alpha^n\varphi}_{L^2(\Omega)}^2\leq \lambda_{n+m-1}^\alpha \norm{U_\alpha^\varphi}_{L^2(\Omega)}^2.
	\end{equation*}
	Now, in view of \eqref{eq:convergence_eigenvalues} we deduce that
	\begin{equation*}
		\norm{\Pi_\alpha^n h_\alpha^\varphi-\Pi_\alpha^n \varphi}_{\mathcal{H}_\alpha}^2\leq C\norm{\varphi}_{L^2(\Omega)}^2\omega_n(\alpha)\quad\text{for $\alpha$ sufficiently large},
	\end{equation*}
	which, combined with \eqref{eq:eigenfunctions_claim1} implies \eqref{eq:eigenfunctions_claim2}.
	
	\noindent\textbf{Step 3:} we claim that
	\begin{equation}\label{eq:eigenfunctions_claim3}
		\norm{h_\alpha^\psi-\frac{\Pi_\alpha^n\psi}{\norm{\Pi_\alpha^n\psi}_{L^2(\Omega)}}}_{\mathcal{H}_\alpha}^2 \leq C\left(\omega_n(\alpha)+\frac{\sqrt{\omega_n(\alpha)}}{\alpha}\right)\quad\text{for $\alpha$ sufficiently large},
	\end{equation}
	where $\psi:=\varphi/\norm{\varphi}_{L^2(\Omega)}$.	In order to see this, we first observe that, by writing everything with respect to the eigenbasis of $\mathcal{E}_\alpha^n$ and using monotonicity of eigenvalues as in the previous step, there holds
	\begin{equation}\label{eq:eigenfunctions_7}
		\norm{\Pi_\alpha^n\psi-\frac{\Pi_\alpha^n\psi}{\norm{\Pi_\alpha^n\psi}_{L^2(\Omega)}}}_{\mathcal{H}_\alpha}^2\leq \lambda_{n+m-1}^\alpha\abs{\norm{\Pi_\alpha^n\psi}_{L^2(\Omega)}-1}^2.
	\end{equation}
	Let us now analyze the right-hand side of the previous inequality. By adding and subtracting $U_\alpha^\psi$ and $\Pi_\alpha^n U_\alpha^\psi$ and using the fact that $\Pi_\alpha^n$ is a projection, we have
	\begin{align*}
		\abs{\norm{\Pi_\alpha^n\psi}_{L^2(\Omega)}-1}\leq \norm{\psi-\Pi_\alpha^n\psi}_{L^2(\Omega)} &\leq \norm{h_\alpha^\psi-\Pi_\alpha^n h_\alpha^\psi}_{L^2(\Omega)}+\norm{U_\alpha^\psi}_{L^2(\Omega)}+\norm{\Pi_\alpha^n U_\alpha^\psi}_{L^2(\Omega)} \\
		&\leq \norm{h_\alpha^\psi-\Pi_\alpha^n h_\alpha^\psi}_{L^2(\Omega)}+2\norm{U_\alpha^\psi}_{L^2(\Omega)}.
	\end{align*}
	Furthermore, \eqref{eq:eigenfunctions_claim1} and the definition of $\omega_n(\alpha)$ yield
	\begin{equation*}
		\abs{\norm{\Pi_\alpha^n\psi}_{L^2(\Omega)}-1}\leq C\sqrt{\omega_n(\alpha)}.
	\end{equation*}
	Combining the previous estimate with \eqref{eq:eigenfunctions_7}, \eqref{eq:convergence_eigenvalues} and \eqref{eq:eigenfunctions_claim2} we get \eqref{eq:eigenfunctions_claim3}. We emphasize that we also used the fact that $\varphi\mapsto h_\alpha^\varphi$ and $\varphi\mapsto\Pi_\alpha^n\varphi$ are linear. Finally, since $\alpha\,\omega_n(\alpha)\to 0$ as $\alpha\to +\infty$, we conclude the proof of the first part. If, in addition, $\Omega$ is of class $C^{1,1}$, then we can use \Cref{lemma:L2} in order to estimate \eqref{eq:eigenfunctions_claim3} and obtain the better remainder term $O(\alpha^{-2})$.
	
\end{proof}

In view of the asymptotic estimates about $U_\alpha^\varphi$ derived in \Cref{sec:quantitative}, we can prove our last main result.
\begin{proof}[Proof of \Cref{cor:eigenfunctions}]
	First of all, by explicit computations we have
	\begin{align*}
		&\norm{\varphi-\frac{\Pi_\alpha^n\varphi}{\norm{\Pi_\alpha^n\varphi}_{L^2(\Omega)}}}_{\mathcal{H}_\alpha}^2-\frac{1}{\alpha}\norm{\partial_{\nnu}\varphi}_{L^2(\partial\Omega)}^2= \norm{\varphi-\frac{\Pi_\alpha^n\varphi}{\norm{\Pi_\alpha^n\varphi}_{L^2(\Omega)}}-U_\alpha^\varphi}_{\mathcal{H}_\alpha}^2\\
		&+\norm{U_\alpha^\varphi}_{\mathcal{H}_\alpha}^2 
		+2\left(\varphi-\frac{\Pi_\alpha^n\varphi}{\norm{\Pi_\alpha^n\varphi}_{L^2(\Omega)}}-U_\alpha^\varphi,U_\alpha^\varphi\right)_{\mathcal{H}_\alpha} 
		-\frac{1}{\alpha}\norm{\partial_{\nnu}\varphi}_{L^2(\partial\Omega)}^2.
	\end{align*}
	Now, we recall that \eqref{eq:minimizer_th1} tells us that
	\begin{equation*}
		\norm{U_\alpha^\varphi}_{\mathcal{H}_\alpha}^2 =\int_{\partial\Omega}U_\alpha^\varphi\partial_{\nnu}\varphi\ds.
	\end{equation*}
	Combining these two facts yields
	\begin{align*}
		\norm{\varphi-\frac{\Pi_\alpha^n\varphi}{\norm{\Pi_\alpha^n\varphi}_{L^2(\Omega)}}}_{\mathcal{H}_\alpha}^2-\frac{1}{\alpha}\norm{\partial_{\nnu}\varphi}_{L^2(\partial\Omega)}^2=&
		\norm{\varphi-\frac{\Pi_\alpha^n\varphi}{\norm{\Pi_\alpha^n\varphi}_{L^2(\Omega)}}-U_\alpha^\varphi}_{\mathcal{H}_\alpha}^2 \\
		&+2\left(\varphi-\frac{\Pi_\alpha^n\varphi}{\norm{\Pi_\alpha^n\varphi}_{L^2(\Omega)}}-U_\alpha^\varphi,U_\alpha^\varphi\right)_{\mathcal{H}_\alpha} \\
		&+\frac{1}{\alpha}\int_{\partial\Omega}(\alpha U_\alpha^\varphi-\partial_{\nnu}\varphi)\partial_{\nnu}\varphi\ds.
	\end{align*}
	We now analyze each of the terms above. For the first one, we use \eqref{eq:eigenfunctions_claim3}, which gives
	\begin{equation*}
		\norm{\varphi-\frac{\Pi_\alpha^n\varphi}{\norm{\Pi_\alpha^n\varphi}_{L^2(\Omega)}}-U_\alpha^\varphi}_{\mathcal{H}_\alpha}^2 \leq C\left(\omega_n(\alpha)+\frac{\sqrt{\omega_n(\alpha)}}{\alpha}\right),
	\end{equation*}
	for some $C>0$ and $\alpha$ sufficiently large (depending on $d$, $\Omega$ and $n$), where $\omega_n(\alpha)$ is as in \eqref{eq:def_omega}. For the second one, we use Cauchy-Schwarz inequality, \eqref{eq:eigenfunctions_claim3}, \eqref{eq:bound_th1} (combined with \eqref{eq:minimizer_th1}) and \eqref{eq:equiv_norms_th1}, which gives
	\begin{equation*}
		\abs{\left(\varphi-\frac{\Pi_\alpha^n\varphi}{\norm{\Pi_\alpha^n\varphi}_{L^2(\Omega)}}-U_\alpha^\varphi,U_\alpha^\varphi\right)_{\mathcal{H}_\alpha}}\leq C\sqrt{\frac{\omega_n(\alpha)}{\alpha}+\frac{\sqrt{\omega_n(\alpha)}}{\alpha^2}}
	\end{equation*}
	for some $C>0$ and $\alpha$ sufficiently large (depending on $d$, $\Omega$ and $n$). Lastly, by definition of $\rho_n(\alpha)$ as in \eqref{eq:def_rho}, we have
	\begin{equation*}
		\frac{1}{\alpha}\abs{\int_{\partial\Omega}(\alpha U_\alpha^\varphi-\partial_{\nnu}\varphi)\partial_{\nnu}\varphi\ds}\leq \rho_n(\alpha).
	\end{equation*}
	Summing up all the contributions, we have 
	\begin{equation*}
		\norm{\varphi-\frac{\Pi_\alpha^n\varphi}{\norm{\Pi_\alpha^n\varphi}_{L^2(\Omega)}}}_{\mathcal{H}_\alpha}^2-\frac{1}{\alpha}\norm{\partial_{\nnu}\varphi}_{L^2(\partial\Omega)}^2=O\left(\sqrt{\frac{\omega_n(\alpha)}{\alpha}+\frac{\sqrt{\omega_n(\alpha)}}{\alpha^2}}\right)+O\left(\rho_n(\alpha)\right),
	\end{equation*}
	as $\alpha\to+\infty$, where the remainder terms only depends on $d$, $\Omega$ and $n$. In view of \Cref{lemma:L2_weak}, \Cref{lemma:L2} and \Cref{lemma:rho} we may conclude the proof.
\end{proof}

Finally, combining \Cref{thm:main} with some explicit computations of Dirichlet eigenvalues and eigenfunctions on rectangles, we are able to prove the following.
\begin{proof}[Proof of \Cref{cor:simplicity}]
	By explicit computations, we have that the eigenvalues on the rectangle $\Omega=(0,l)\times(0,L)$ are given by
	\begin{equation*}
		\lambda_{n,m}:=\left(\frac{n\pi}{l}\right)^2+\left(\frac{m\pi}{L}\right)^2,
	\end{equation*}
	and a corresponding $L^2(\Omega)$-orthonormal family of eigenfunctions is given by 
	\begin{equation*}
		\varphi_{n,m}(x,y):=\frac{2}{\sqrt{lL}}\sin\left(\frac{n\pi}{l}x\right)\sin\left(\frac{m\pi}{L}y\right),
	\end{equation*}
	for $n,n\geq 1$. We also denote by $\{\lambda_{n,m}^\alpha\}_{n,m\in\N}$ the corresponding Robin eigenvalues, so that
	\begin{equation*}
		\lambda_{n,m}^\alpha\to \lambda_{n,m}\quad\text{as }\alpha\to+\infty,~\text{for all }n,m\in\N.
	\end{equation*}
	 We emphasize that, for the sake of simplicity, in this proof, we adopt the notation with two indices $(n,m)$ for eigenvalues and eigenfunctions, which does not correspond with the notation with a single index used in the rest of the paper.  If $\lambda_{n,m}$ is simple, we immediately get the thesis, since by continuity $\lambda_{n,m}^\alpha$ is simple for $\alpha$ sufficiently large. Hence, we assume $\lambda_{n,m}$ to be multiple. More precisely, since one can easily check that $\lambda_{n,m}=\lambda_{i,j}$ if and only if either $n=i$ and $m=j$ (same index) or $n\neq i$ and $m\neq j$ (both indices change), we assume the latter, i.e.
	\begin{equation}\label{eq:rect_2}
		\frac{n^2}{l}+\frac{m^2}{L}=\frac{i^2}{l}+\frac{j^2}{L}
	\end{equation}
	 and
	 \begin{equation}\label{eq:rect_3}
	 	n\neq i\quad\text{and}\quad m\neq j.
	 \end{equation}
 	In order to prove the result, we want to make use of \Cref{thm:main}; rectangles are Lipschitz domains, hence let us verify \eqref{eq:hp_eigen}. We claim that
	\begin{equation}\label{eq:rect_1}
		\int_{\partial\Omega}\partial_{\nnu}\varphi_{n,m}\,\partial_{\nnu}\varphi_{i,j}\ds= 0
	\end{equation}
	which follows by direct computations. Indeed, one can split
	\begin{align*}
		\int_{\partial\Omega}\partial_{\nnu}\varphi_{n,m}\,\partial_{\nnu}\varphi_{i,j}\ds&=\int_0^L\partial_x\varphi_{n,m}(0,y)\partial_x\varphi_{i,j}(0,y)\d y +\int_0^L\partial_x\varphi_{n,m}(l,y)\partial_x\varphi_{i,j}(l,y)\d y \\
		&+\int_0^l\partial_y\varphi_{n,m}(x,0)\partial_y\varphi_{i,j}(x,0)\dx+\int_0^l\partial_y\varphi_{n,m}(x,L)\partial_y\varphi_{i,j}(x,L)\dx
	\end{align*}
	and then explicitly compute each term, which yields
	\begin{align*}
		&\int_{\partial\Omega}\partial_{\nnu}\varphi_{n,m}\,\partial_{\nnu}\varphi_{i,j}\ds \\
		&=\left(\frac{2n\pi}{l\sqrt{lL}}\right)^2\left[\int_0^L\sin\left(\frac{m\pi}{L}y\right)\sin\left(\frac{j\pi}{L}y\right)\dy+(-1)^{m+j}\int_0^L\sin\left(\frac{m\pi}{L}y\right)\sin\left(\frac{j\pi}{L}y\right)\dy\right] \\
		&+\left(\frac{2m\pi}{L\sqrt{lL}}\right)^2\left[\int_0^l\sin\left(\frac{n\pi}{l}x\right)\sin\left(\frac{i\pi}{l}x\right)\dx+(-1)^{n+i}\int_0^l\sin\left(\frac{n\pi}{l}x\right)\sin\left(\frac{i\pi}{l}x\right)\dx\right].
	\end{align*}
	Then, since $n\neq i$ and $m\neq j$, one can easily see that each integral vanishes, thus proving \eqref{eq:rect_1}. This means that any rectangle of $\R^2$ satisfies the assumptions of \Cref{thm:main} with this choice of eigenfunctions, which satisfy \eqref{eq:hp_eigen}. At this point, we claim that
	\begin{equation}\label{eq:rect_4}
		\int_{\partial\Omega}(\partial_{\nnu}\varphi_{n,m})^2\ds\neq \int_{\partial\Omega}(\partial_{\nnu}\varphi_{i,j})^2\ds.
	\end{equation}
	In view of the previous computations, one can find that
	\begin{equation*}
		\int_{\partial\Omega}(\partial_{\nnu}\varphi_{n,m})^2\ds=4\pi^2\left(\frac{m^2}{L^3}+\frac{n^2}{l^3}\right).
	\end{equation*}
	So, if by contradiction
	\begin{equation*}
				\int_{\partial\Omega}(\partial_{\nnu}\varphi_{n,m})^2\ds=\int_{\partial\Omega}(\partial_{\nnu}\varphi_{i,j})^2\ds,
	\end{equation*}
	this would imply that
	\begin{equation*}
		\frac{m^2}{L^3}+\frac{n^2}{l^3}=\frac{j^2}{L^3}+\frac{i^2}{l^3}.
	\end{equation*}
	Combining this identity with \eqref{eq:rect_2}, we would obtain that 
	\begin{equation*}
		\frac{m^2}{L^2}\left(\frac{1}{l}-\frac{1}{L}\right)=\frac{j^2}{L^2}\left(\frac{1}{l}-\frac{1}{L}\right)
	\end{equation*}
	which would imply that $m=j$, since $L\neq l$ by assumption. But this contradicts \eqref{eq:rect_3} and so \eqref{eq:rect_4} holds. Now, in view of \Cref{thm:main} we have that $\lambda_{n,m}^\alpha\neq \lambda_{i,j}^\alpha$, for $\alpha$ sufficiently large. This completes the proof.
	
\end{proof}

\appendix

\section{Regularity of solutions}\label{sec:appendix}

In this appendix, we recall some regularity results for the solutions of the PDEs appearing in the present work, i.e. Laplacian eigenfunctions, torsion functions and harmonic extensions. Furthermore, we present a justification of the divergence theorem, which is suitable for our framework. We remark that some of these facts are classical and can be found in many sources. However, we prefer to refer primarily to \cite{BGM2022}, which is a fairly comprehensive treatise and provides the appropriate bibliographical context. We recall that $\Omega\sub\R^d$ is an open, bounded Lipschitz set.
\subsection{Regularity of Dirichlet eigenfunctions}\label{subsec:reg_eigen}
If $\varphi\in H^1_0(\Omega)$ is an eigenfunction of the Dirichlet Laplacian for some eigenvalue $\lambda$, that is
\begin{equation*}
	\begin{bvp}
		-\Delta \varphi &=\lambda\varphi, &&\text{in }\Omega, \\
		\varphi&=0, &&\text{on }\partial\Omega,
	\end{bvp}
\end{equation*}
then by elliptic regularity theory we have that $\varphi$ is analytic in $\Omega$. Moreover, being $\Omega$ a Lipschitz domain, the exterior cone condition is satisfied and so we have that $\varphi\in C^{0,\gamma}(\overline{\Omega})$, for some $\gamma\in(0,1)$, see e.g. \cite[Theorem 8.29]{GT2001} and the remarks below. Moreover, from \cite[Corollary 3.7, (ii)]{BGM2022} we know that $\varphi\in H^{3/2}(\Omega)$, while from \cite[Corollary 3.7, (iv)]{BGM2022} we derive that $\partial_{\nnu}\varphi\in L^2(\partial\Omega)$, i.e. \eqref{ass:Omega} holds.

If we ask further assumptions on the domain, the Dirichlet eigenfunctions consequently gain further regularity. More precisely, if $\Omega$ satisfies a uniform outer ball condition, which is true e.g. if $\Omega$ is convex or of class $C^{1,1}$, then $\varphi\in H^2(\Omega)$ (see \cite[Theorem 1.1]{adolfsson} and \cite[Theorem 2.4.2.5]{grisvard}). Moreover, if $\Omega$ is of class $C^{k+1,1}$, then $\varphi\in H^{k+2}(\Omega)$, for any $k\geq 1$, see \cite[Theorem 2.5.1.1]{grisvard}.

\subsection{Regularity of Robin eigenfunctions and torsion functions}\label{subsec:reg_robin_torsion}

If $\varphi\in H^1(\Omega)$ is a Robin eigenfunction (with parameter $\alpha>0$) for some eigenvalue $\lambda$, that is
\begin{equation*}
	\begin{bvp}
		-\Delta \varphi&=\lambda\varphi, &&\text{in }\Omega, \\
		\partial_{\nnu}\varphi+\alpha\varphi&=0, &&\text{on }\partial\Omega,
	\end{bvp}
\end{equation*}
then it is known that $\varphi$ is analytic in $\Omega$ and that $\varphi\in C^{0,\gamma}(\overline{\Omega})$, for some $\gamma\in(0,1)$, see e.g. \cite[Theorem 3.14]{nittka}. For what concerns Sobolev regularity, from \cite[Corollary 5.7]{BGM2022} we can deduce that $\varphi\in H^{3/2}(\Omega)$ and that $\partial_{\nnu}\varphi\in L^2(\partial\Omega)$. Let us quickly prove this fact. From \cite[Corollary 5.7]{BGM2022} (with $s=3/2$) we have that there exists $\psi\in H^{3/2}(\Omega)$ such that $\Delta\psi\in L^2(\Omega)$ and $\partial_{\nnu}\psi=-\alpha\varphi$. Now let $\widetilde{\varphi}:=\varphi-\psi$. By definition, we know that $\widetilde{\varphi}\in H^{1/2}(\Omega)$, that $-\Delta\widetilde{\varphi}=\lambda\varphi+\Delta\psi\in L^2(\Omega)$ and that $\partial_{\nnu}\widetilde{\varphi}=0$ on $\partial\Omega$; hence, by \cite[Corollary 5.7, (ii)]{BGM2022} (with $s=1/2$) it follows that $\widetilde{\varphi}\in H^{3/2}(\Omega)$, which readily implies $\varphi\in H^{3/2}(\Omega)$. Finally, again \cite[Corollary 5.7]{BGM2022} provides that $\partial_{\nnu}\varphi\in L^2(\partial\Omega)$.

Furthermore, reasoning analogously, one can prove that if $f\in L^2(\partial\Omega)$ and $U_{\partial\Omega,\alpha,f}\in H^1(\Omega)$ is the corresponding torsion function, i.e. the function achieving $T_\alpha(\partial\Omega,f)$ and satisfying
\begin{equation*}
	\begin{bvp}
		-\Delta U_{\partial\Omega,\alpha,f}&=0, &&\text{in }\Omega, \\
		\partial_{\nnu}U_{\partial\Omega,\alpha,f}+\alpha U_{\partial\Omega,\alpha,f}&=f, &&\text{on }\partial\Omega,
	\end{bvp}
\end{equation*}
then $U_{\partial\Omega,\alpha,f}\in H^{3/2}(\Omega)$ and $\partial_{\nnu}U_{\partial\Omega,\alpha,f}\in L^2(\partial\Omega)$. In particular, the boundary conditions can be intended as identities in $L^2(\partial\Omega)$.

\subsection{Harmonic extensions}\label{subsec:harm_ext}
It is known that for any $f\in H^{1/2}(\partial\Omega)$ there exists a unique function $V_f\in H^1(\Omega)$ satisfying
\begin{equation}\label{eq:harm_ext}
	\begin{bvp}
		-\Delta V_f&=0, &&\text{in }\Omega, \\
		V_f&=f, &&\text{on }\partial\Omega.
	\end{bvp}
\end{equation}
For reference, one can see for instance \cite[Theorem 3.6, (i)]{BGM2022} (with $s=1$). Moreover, one has the following estimate:
\begin{equation*}
	\norm{V_f}_{H^1(\Omega)}\leq C\norm{f}_{H^{1/2}(\partial\Omega)},
\end{equation*}
for some $C>0$ depending only on $d$ and $\Omega$. Actually, more in general, from \cite[Theorem 3.6, (i)]{BGM2022} we have that for any $s\in\left[\frac{1}{2},\frac{3}{2}\right]$ and any $f\in H^{s-1/2}(\partial\Omega)$ there exists $V_f\in H^s(\Omega)$ satisfying \eqref{eq:harm_ext} and
\begin{equation*}
	\norm{V_f}_{H^s(\Omega)}\leq C\norm{f}_{H^{s-1/2}(\partial\Omega)},
\end{equation*}
for some $C>0$ depending only on $d$ and $\Omega$. Finally, if $\Omega$ is of class $C^{1,1}$ and $f\in H^{3/2}(\partial\Omega)$, then $V_f\in H^2(\Omega)$, see e.g. \cite[Theorem 2.4.2.5]{grisvard}.

\subsection{Divergence theorem}\label{subsec:dive}
Let us consider the space
\begin{equation*}
	H_{\tu{div}}(\Omega):=\{\vec{G}\in [L^2(\Omega)]^d\colon \dive \vec{G}\in L^2(\Omega)\}.
\end{equation*}
It is known (see e.g. \cite[Chapter 20]{Tartar2007} and \cite[Section III.2]{Galdi2011}) that $H_{\tu{div}}(\Omega)$ is continuously embedded in $H^{-1/2}(\partial\Omega)$ and that, if $\vec{F}\in H_{\tu{div}}(\Omega)$, then the following generalization of the divergence theorem holds:
\begin{equation*}
	\int_\Omega \nabla u\cdot\vec{F}\dx =-\int_\Omega u\dive \vec{F}\dx+_{H^{-1/2}(\partial\Omega)}\langle \vec{F}\cdot\nnu,u \rangle_{H^{1/2}(\partial\Omega)}\quad\text{for all }u\in H^1(\Omega).
\end{equation*}
Moreover, if $\vec{F}$ is sufficiently regular, then the duality product in the identity above can be actually written as an integral.
In particular, let us consider \cite[Corollary 4.5]{BGM2022} with $\epsilon=1/2$ and $\vec{F}\in [H^{1/2}(\Omega)]^d$ such that $\Delta\vec{F}\in [H^{-1}(\Omega)]^d$ and $\dive\vec{F}\in [L^2(\Omega)]^d$. Then by \cite[Corollary 4.5]{BGM2022} there holds
\begin{equation*}
	\int_\Omega \nabla u\cdot\vec{F}\dx =-\int_\Omega u\dive \vec{F}\dx+\int_{\partial\Omega} \vec{F}\cdot\nnu\,u\ds\quad\text{for all }u\in H^1(\Omega).
\end{equation*}
In particular, if $\vec{F}=\nabla v$, for an arbitrary $v\in H^{3/2}(\Omega)$ such that $\Delta v\in L^2(\Omega)$, we have
\begin{equation}\label{eq:A_int_by_parts}
	\int_\Omega \nabla u\cdot\nabla v\dx=-\int_\Omega u\Delta v\dx+\int_{\partial\Omega}u\,\partial_{\nnu}v\ds\quad\text{for all }u\in H^1(\Omega).
\end{equation}

\section*{Acknowledgments}
The author is supported by the project ERC VAREG - \emph{Variational approach to the regularity of the free boundaries} (grant agreement No. 853404). The author also acknowledges the MIUR Excellence Department Project awarded to the Department of Mathematics, University of Pisa, \texttt{CUP\_I57G22000700001} and support from the 2024 INdAM-GNAMPA project no. \texttt{CUP\_E53C23001670001}. The author thanks G. Bevilacqua, G. Tortone and B. Velichkov for useful discussions. The author also warmly thanks the anonymous referees for the careful reading and the numerous valuable suggestions, which helped to improve the quality of the paper.

\bigskip

\noindent\textbf{Data availability.} Data sharing not applicable to this article as no datasets were generated or analyzed during
the current study.
	
\bibliographystyle{aomalpha}

\bibliography{biblio}

\end{document}